\documentclass[english]{amsart}
 
\usepackage{graphicx}
\usepackage[justification=justified,font=small]{caption}
\usepackage{bbm}
 \usepackage{color}
\numberwithin{equation}{section} %% Comment out for sequentially-numbered
\numberwithin{figure}{section} %% Comment out for sequentially-numbered
\newtheorem{thm}{Theorem}
\newtheorem{cor}{Corollary}[section]
\newtheorem{lem}[cor]{Lemma}
 \newtheorem{prop}[cor]{Proposition}
 \newcommand\proc{\noindent\textit}

\newcommand{\1}{\mathbbm{1}}

\newcommand{\e}{\mathrm e}
\DeclareMathOperator{\Int}{Int}

\renewcommand{\phi}{\varphi}

\DeclareMathOperator{\dom}{dom}

\def\R{\mathbb{R}}

\def\N{\mathbb{N}}

\def\U{\mathcal{U}}

\begin{document}
\title[Strong renewal theorems and Lyapunov spectra]
{Strong renewal theorems and Lyapunov spectra for
$\boldsymbol\alpha$-Farey and $\boldsymbol\alpha$-L\"uroth systems}

\author{Marc Kesseb{\"o}hmer}

\address{Fachbereich 3 -- Mathematik und Informatik, Universit\"at Bremen, Bibliothekstr.
1, 28359 Bremen, Germany}
\email{mhk@math.uni-bremen.de}

\author{Sara Munday}

\address{Mathematical Institute, University of St. Andrews, North Haugh, St.
Andrews KY16 9SS, Scotland}

\email{sam20@maths.st-and.ac.uk}

\author{Bernd O. Stratmann}

\address{Fachbereich 3 -- Mathematik und Informatik, Universit\"at Bremen, Bibliothekstr.
1, 28359 Bremen, Germany}

\email{bos@math.uni-bremen.de}


\date{\today}

\begin{abstract}
In this paper we introduce and study the $\alpha$-Farey
map  and  its associated jump transformation, the
$\alpha$-L\"uroth map, for an arbitrary countable partition $\alpha$
of the unit interval with atoms which accumulate only at the origin.
These maps represent linearised generalisations of the Farey map
and the Gauss map from elementary number theory.
First, a thorough analysis of  some of their
topological and ergodic-theoretic properties is given, including
establishing exactness for both types of these maps.
The first main result then is to establish weak and strong renewal laws for
what we have called $\alpha$-sum-level sets for the
$\alpha$-L\"uroth map.  Similar results have previously been obtained for the
Farey map
and the Gauss map, by using infinite ergodic theory. In this respect, a side
product  of  the paper is to
allow for greater transparency of some of the core ideas of infinite ergodic
theory. The second remaining result is to  obtain a
complete description of the Lyapunov spectra of the
$\alpha$-Farey
map  and  the
$\alpha$-L\"uroth map
in terms of the thermodynamical formalism. We show how to derive these spectra,
and then give various examples which demonstrate the diversity of their
behaviours in dependence on the chosen partition $\alpha$.
\end{abstract}
\maketitle

\section{Introduction and statement of results}
In this paper we consider the {\em $\alpha$-Farey map}
$F_{\alpha}:\U \to \U$, which is
given for a countable partition  $\alpha:=\{A_n:n\in\N\}$  of
the unit interval $\mathcal{U}:=[0,1]$  by
\[
F_{\alpha}(x):=\left\{
        \begin{array}{ll}
          (1-x)/a_1, & \hbox{if $x\in A_1$,} \\
          {a_{n-1}}(x-t_{n+1})/a_{n}+t_n, & \hbox{if $x\in A_n$, for  $n\geq2$,}
\\
	  0, & \hbox{if $x=0$, }
        \end{array}
      \right.
\]
where $a_n$ is equal to the Lebesgue measure $\lambda(A_n)$ of the
atom $A_n\in\alpha$, and  $t_n:=\sum_{k=n}^\infty a_k$
denotes the Lebesgue measure of the $n$-th tail of $\alpha$. (It is
assumed throughout that $\alpha$ is a countable partition of $\U$ consisting of left open, right closed intervals; also,  we always assume that the atoms of $\alpha$ are ordered from right to left,
starting with $A_1$,  and that these atoms accumulate only at the origin.) Similarly to the way in which the Gauss
map coincides with
the jump transformation of  the Farey map with respect to the
interval $(1/2,1]$, one finds that the map $F_{\alpha}$ gives rise to
the jump transformation $L_{\alpha}$ with respect to
the interval $A_{1}$. It turns out that for the \textit{harmonic partition}
$\alpha_H$, given by $a_{n}:= 1/(n(n+1))$, we have
that the so-obtained jump transformation $L_{\alpha_H}$
coincides with the {\em alternating
L\"{u}roth map} (see \cite{KKK}). For a general partition $\alpha$, we therefore refer to
 $L_{\alpha}$ as the  {\em $\alpha$-L\"{u}roth map}, and we will
 see that this map
is explicitly given by
\[
L_{\alpha}(x):=
\left\{
	\begin{array}{ll}
	    ({t_n-x})/a_n, & \text{ if }x\in A_n,\ n\in\N,\\
	  0, & \hbox{ if } x=0.
	\end{array}
      \right.
\]

Note that this type of generalised L\"uroth map has also been
investigated, amongst others,  in
\cite{BBDK} and \cite{DK1}. Also, a class of maps
very  similar to
our class of $\alpha$-Farey maps has been considered in \cite{XJW}.
\begin{figure}[ht!]
%\captionsetup{justification=justified}
\begin{center}
\includegraphics[width=0.4\textwidth]{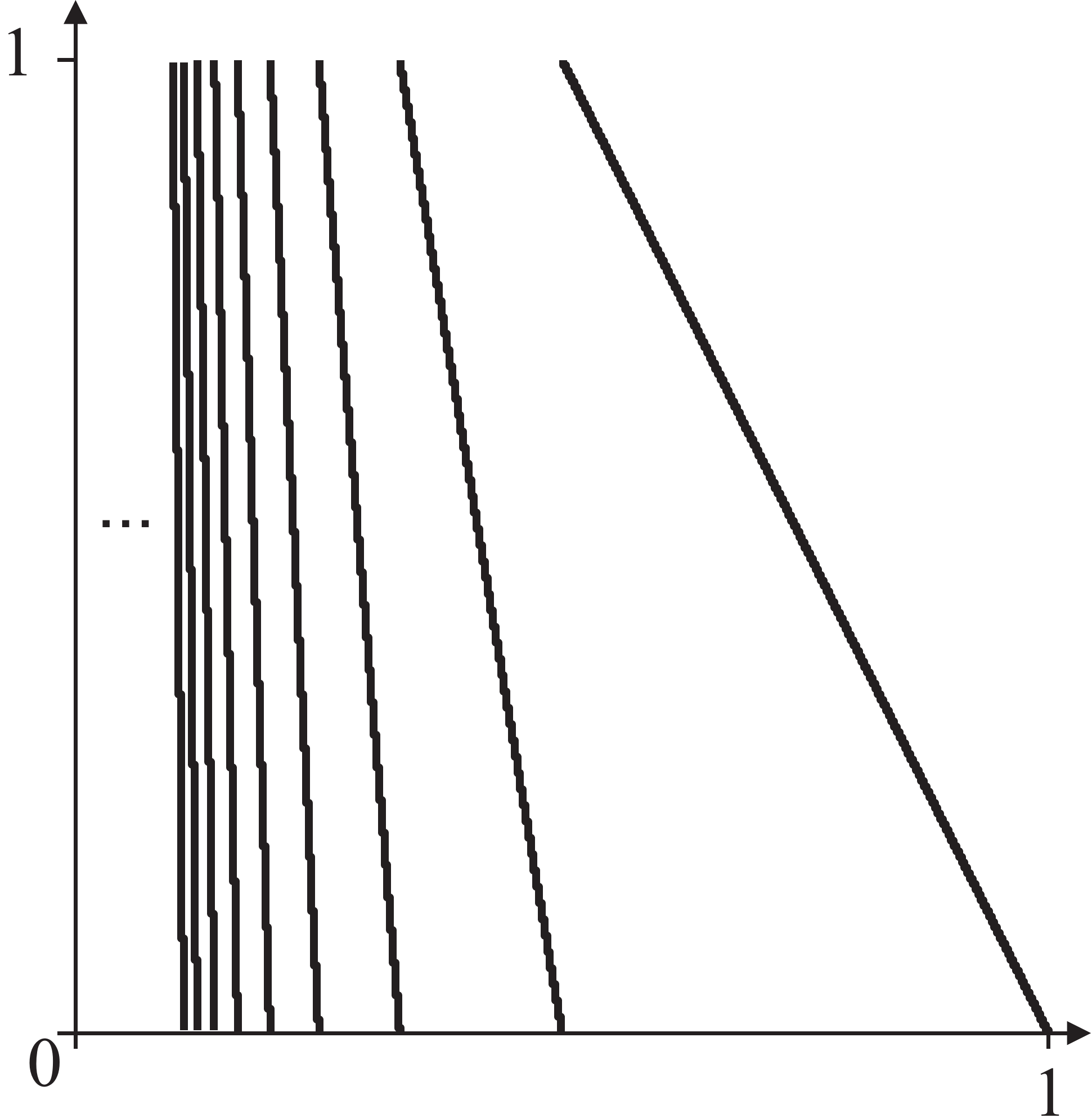}\hspace{0.1\textwidth}
\includegraphics[width=0.4\textwidth]{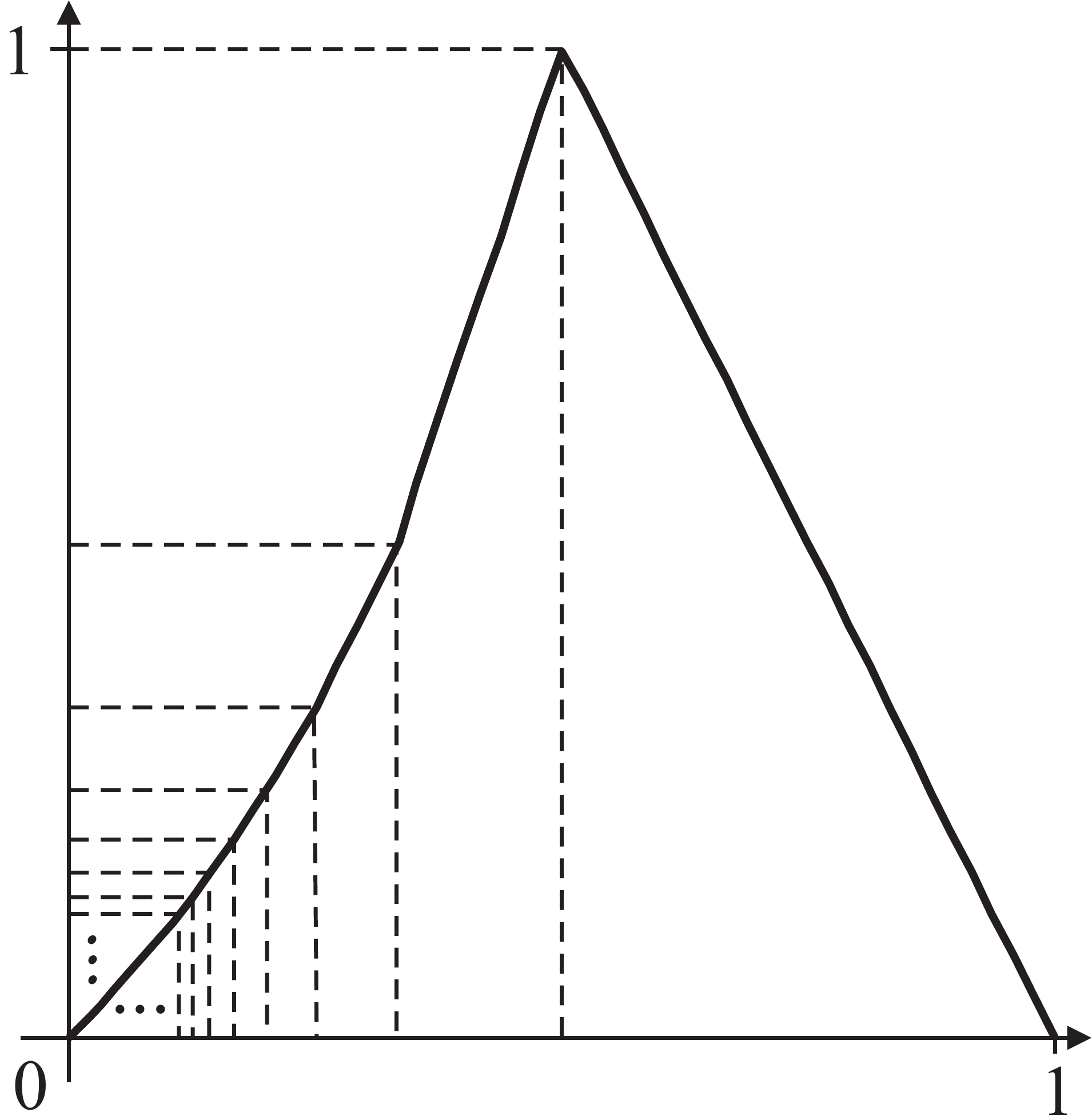}
\caption{The alternating  L\"uroth and $\alpha_H$-Farey map, where $t_{n}=1/n$, $n\in \N$.}\end{center}
\end{figure}\label{fig:ClassicalLF}

The main goal of this paper is to give a thorough analysis of the
two maps $F_{\alpha}$ and $L_{\alpha}$.
This includes the  study of the sequence of $\alpha$-sum-level
sets $\left(\mathcal{L}^{(\alpha)}_n\right)_{n\in \N}$ arising from
the $\alpha$-L\"uroth map, for an arbitrary given partition $\alpha$.
These sets
are defined by
\[
\mathcal{L}^{(\alpha)}_n:=\left\{x \in C_{\alpha}(\ell_1, \ell_2, \ldots,
\ell_k):\sum_{i=1}^k \ell_i=n, \hbox{ for some }  k\in\N\right\},
\]
where
$C_{\alpha}(\ell_1, \ell_2, \ldots,
\ell_k):=\{x\in\U:L_\alpha^{i-1}(x)\in A_{l_i}, \text{ for  all } i=1, \ldots, k\}$
denotes a cylinder set arising from the map $L_{\alpha}$.  The sets
$\mathcal{L}^{(\alpha)}_n$ can also be written dynamically in terms of
$F_{\alpha}$, that is, one  immediately verifies that
$\mathcal{L}^{(\alpha)}_n= F_{\alpha}^{-(n-1)}(A_{1})$, for all $n \in
\N$.

Throughout,
$\alpha$  is said to be of
{\em finite type} if
for the tails $t_{n}$ of $\alpha$ we have that $\sum_{n=1}^{\infty}
t_{n}$ converges, otherwise $\alpha$  is said to be of
\textit{infinite type}.
Moreover,  a partition $\alpha$ is called {\em expansive of exponent $\theta$} if
its tails
satisfy the power law $t_{n} = \psi(n) n^{-\theta}$,
for all $n \in \N$, for some
$\theta \geq 0$ and for some
slowly varying\footnote{A measurable function $f:\R^{+} \to \R^{+}$ is said
to be
{\em slowly varying}  if $
\lim_{x\to\infty}f(x y)/f(x)=1$, for all $y>0$.} function $\psi$.
Note that in this situation we have that $\lim_{n\to\infty}
t_{n}/t_{n+1}=1$, and hence the right derivative of $F_{\alpha}$ at
zero is equal to $1$, which explains  why  this type of partition
is referred to as expansive.

Also, a partition
$\alpha$ is said to be \textit{expanding} if
$\lim_{n\to \infty} t_{n}/t_{n+1}= \rho$, for some $\rho>1$.
In this situation we have that the right derivative of $F_{\alpha}$ at
zero is equal to $\rho$, and that  is why  we refer to it as
expanding (cf. Lemma \ref{eq:asympExpansive} (2)). Clearly, if $\alpha $ is expanding, then  $F_\alpha$ is of finite type.
Furthermore,  a partition $\alpha$ is called {\em eventually decreasing} if
$a_{n+1}\leq a_n$, for all $n\in \N$ sufficiently large.

Throughout, we use the notation $a_n\sim b_n$ to denote $\lim_{n\to\infty} a_n/b_n=1$.
 \begin{thm}[Renewal laws for sum-level sets]\label{renewal}\hspace{3cm}
  \begin{enumerate}\item  For the Lebesgue measure $\lambda(\mathcal{L}^{(\alpha)}_n)$ of the
    $\alpha$-sum-level sets of a given partition $\alpha$ of $\U$
    we have that $\sum_{n=1}^{\infty}
\lambda(\mathcal{L}^{(\alpha)}_n)$
diverges, and that
       \[
       \lim_{n\to\infty}\lambda
       \left(\mathcal{L}^{(\alpha)}_n\right)
       =\left\{
\begin{array}{ll}
0 ,& \hbox{if   $\alpha$  is of
infinite type;}\\
\left(\sum_{k=1}^\infty t_{k}\right)^{-1}, &  \hbox{if   $\alpha$  is of
finite type.}
\end{array}
\right.	\]
\item
For a given  partition
$\alpha$ which is either    expansive of exponent $\theta\in [0,1]$ or  of finite type,
we have  the following
estimates for the asymptotic behaviour of
$\lambda(\mathcal{L}^{(\alpha)}_n)$.
\begin{enumerate}
   \item[ $(i)$] \textsc{Weak renewal law.}   With $K_{\alpha}:=(\Gamma(2-\theta)
   \Gamma(1+\theta))^{-1}$ for  $\alpha$ expansive of exponent $\theta\in [0,1]$, and with $K_\alpha:=1$ for $\alpha$  of finite type, we have that
\[ \sum_{k=1}^{n} \lambda\left(\mathcal{L}^{(\alpha)}_k\right)
\sim K_{\alpha} \, \cdot n \cdot\left(\sum_{k=1}^{n}
 t_{k}\right)^{-1}.\]
\item[$(ii)$] \textsc{Strong renewal law.}
 With $k_{\alpha}:=(\Gamma(2-\theta)
   \Gamma(\theta))^{-1}$ for  $\alpha$  expansive of exponent $\theta\in (1/2,1]$, and with $k_\alpha:=1$ for $\alpha$  of finite type, we have
\[ \lambda\left(\mathcal{L}^{(\alpha)}_{n}\right)
\sim k_{\alpha} \cdot
\left(\sum_{k=1}^{n} t_{k}\right)^{-1}.
\]
\end{enumerate}
\end{enumerate}

\end{thm}

\proc{Remark 1.}  Note
    that, by using a result of Garsia and Lamperti (\cite{GL}), we have for an expansive partition $\alpha$ of exponent
   $\theta \in (0,1)$, that
   \[ \liminf _{n\to \infty}  \left( n\cdot t_n \cdot \lambda
   \left(\mathcal{L}^{(\alpha)}_{n}\right)\right)=
   \frac{\sin \pi
   \theta}{\pi}.\]
   Moreover,  if
   $\theta \in (0,1/2)$, then the corresponding limit does not exist
      in general. However, in this situation the existence of
      the limit is  always guaranteed at least  on the complement
      of some set of integers of
    zero  density\footnote{The density of a set of integers $A$ is given, where the limit exists, by $d(A)=\lim_{n\to \infty} \#A(n)/n$, where $A(n):=\{1, \ldots, n\}\cap A$.}.
   \medbreak

    \begin{figure}[h!]\begin{center}
     \includegraphics[width=0.4\textwidth]{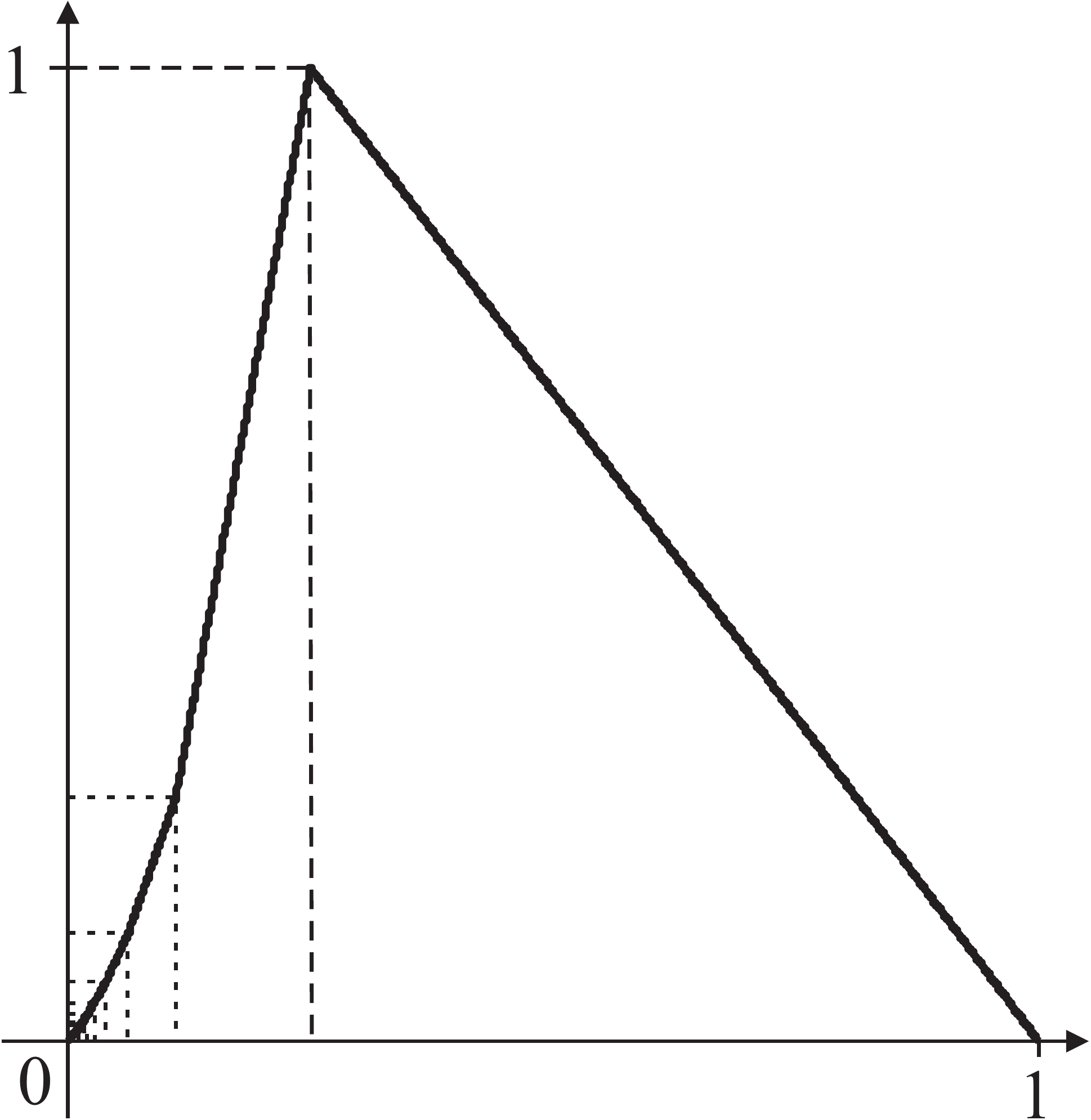}
    \hspace{0.1\textwidth}
    \includegraphics[width=0.4\textwidth]{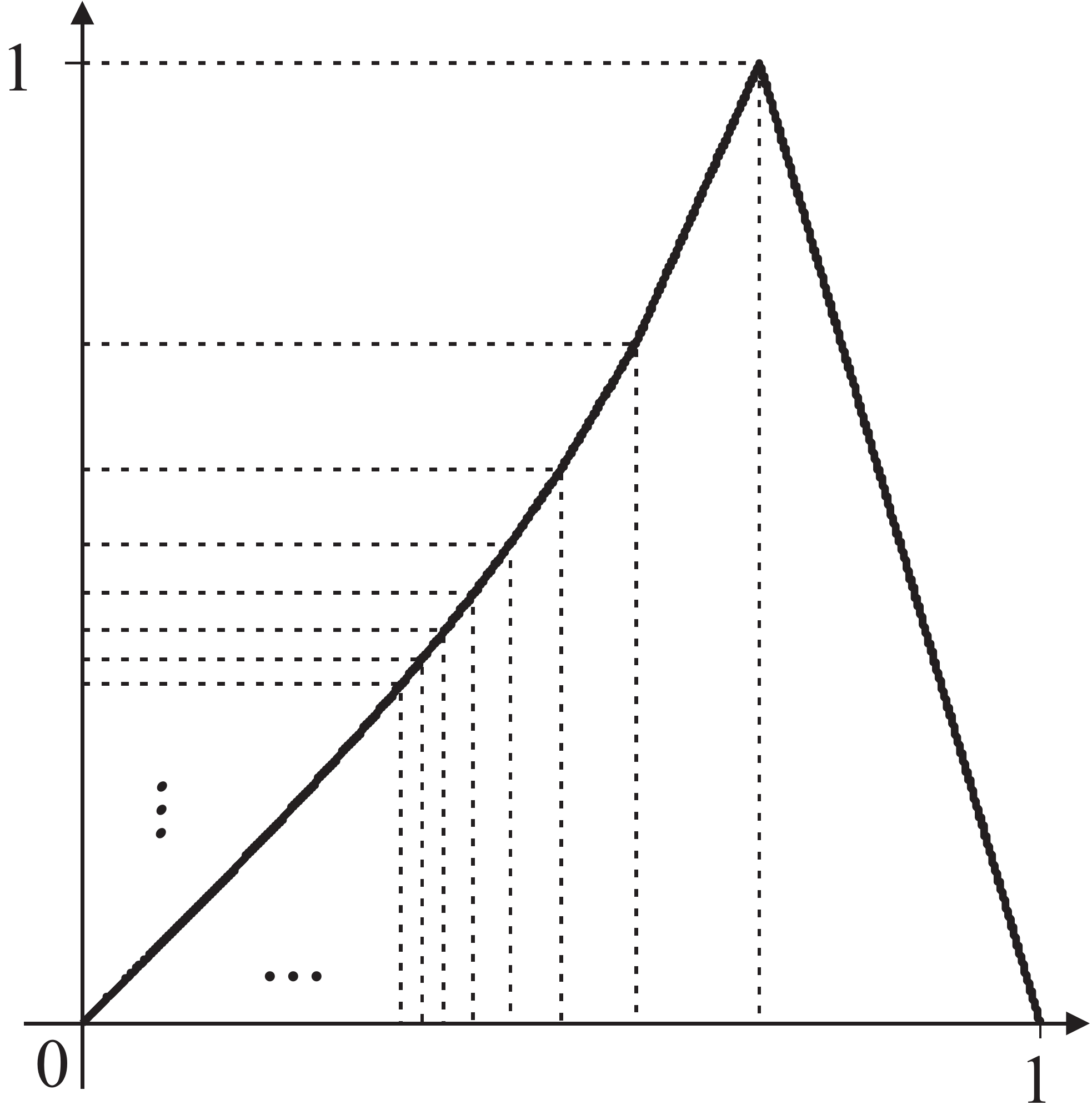}
    \caption{\label{fig:expansiveSystems} The graphs of two $\alpha$-Farey maps  with $\alpha$ expansive.
    The partition on the left is of finite type with  $t_n=1/n^{2}$,
    $n\in \N$ and the partition on the
    right is of infinite type with $t_n={1/\sqrt{n}}$, $n\in \N$.}\end{center}
    \end{figure}

In order to state our remaining main results, recall that
the Lyapunov exponent of a differentiable map $S:\mathcal{U}\to \mathcal{U}$ at a
point $x\in\mathcal{U}$ is defined, provided the limit exists, by
\[
\Lambda(S, x):=\lim_{n\to \infty}\frac1n
\sum_{k=0}^{n-1}\log|S'(S^k(x))|.
\]
Our second main theorem gives a complete fractal-geometric description of the
Lyapunov spectra associated with the map
$L_{\alpha}$. That is, we consider the spectral sets
$\{s\in \R :\{x\in\mathcal{U}:\Lambda(L_{\alpha}, x)=s\} \not=\emptyset\}$
associated with the Hausdorff dimension function  $\tau_\alpha$, which is given by
\[
\tau_\alpha(s):=\dim_H(\{x\in\mathcal{U}:\Lambda(L_{\alpha}, x)=s\}).
\]
 In the following $p: \R \to \R \cup \{\infty\}$  denotes the {\em $\alpha$-L\"uroth pressure function},  given by
$p(u):= \log
\sum_{n=1}^{\infty}a_{n}^{u}$.
 We say that {\em $L_\alpha$ exhibits no
phase transition} if and
only if the pressure function $p$ is differentiable everywhere (that is,
the right and left derivatives of $p$ coincide everywhere, with the
convention that $p'(u)=\infty$ if
$p(u)=\infty$). We refer to  \cite{Sarig} for an interesting further discussion of the phenomenon of phase transition in the context of  countable state Markov chains.

\begin{thm}[Lyapunov spectrum for $\alpha$-L\"uroth systems]\label{multi1}
For a given partition $\alpha$,  the Hausdorff dimension function
of the Lyapunov spectrum associated with $L_{\alpha}$
is given as follows. For $t_{-}:=\min\{-\log a_{n}:n\in\N\}$
we have that $\tau_{\alpha}$ vanishes on $(-\infty,t_{-})$, and
for each $s\in(t_{-},\infty)$ we have \[
\tau_{\alpha}(s)=\inf_{u\in\R}\left(u+s^{-1} p(u) \right).\]
 Moreover, $\tau_{\alpha}(s)$ tends to $t_{\infty}:=\inf\{r>0:\sum_{k=1}^{\infty}a_{n}^{r}<\infty\}\leq1$
for $s$ tending to infinity. Note that $t_{\infty}$ is also equal
to the Hausdorff dimension of the Good-type
set $G_\infty^{(\alpha)}$ associated to $L_{\alpha}$, given by
\[
G_\infty^{(\alpha)}:=\{[\ell_1, \ell_2, \ldots]_\alpha:\lim_{n\to\infty}\ell_n=\infty\}.
\]
Concerning the possibility of phase transitions for $L_{\alpha}$, the following
hold:
\begin{itemize}
\item If $\alpha$ is expanding, then $L_{\alpha}$ exhibits no phase
transition and $t_{\infty} = 0$.
\item If $\alpha$ is expansive of exponent
$\theta>0$ and eventually decreasing,
then $L_{\alpha}$ exhibits no phase
transition if and only if $\sum_{n=1}^{\infty}
\psi(n)^{1/(1+\theta)} (\log n)/n$ diverges. Moreover,
in this situation we have  that $t_{\infty}=1/(1+\theta)$.
\item If $\alpha$ is expansive of exponent
$\theta=0$, then $L_{\alpha}$ exhibits no phase
transition if and only if $\sum_{n=1}^\infty a_n\log(a_n)$ diverges.
Moreover, in this situation we have that $t_\infty=1$.
\end{itemize}
\end{thm}

Note that the Lyapunov spectra for the Gauss map and the Farey map  have been determined
 in \cite{KesseboehmerStratmann:07}.  Also, the sets $G_\infty^{(\alpha)}$ are named for I.J. Good \cite{Good}, for his results concerning similar sets 
 in the continued fraction setting.
  
In our final main theorem we consider the Lyapunov spectra arising from the  maps $F_{\alpha}$. In other words,
 we consider the spectral sets $\{s\in \R: \{x\in\mathcal{U}:\Lambda(F_{\alpha},
x)=s\}\not=\emptyset\}$ associated with the Hausdorff dimension-function  $\sigma_\alpha(s)$, given by
\[
\sigma_\alpha(s):=\dim_H(\{x\in\mathcal{U}:\Lambda(F_{\alpha},
x)=s\}).\]
We define the \emph{$\alpha$-Farey free
energy function}  $v:\R\to\R$, to be given by
\[
v(u):=\inf\left\{
r\in\R:\sum_{n=1}^{\infty}a_{n}^{u}\exp\left(-rn\right)\leq1\right\} .\]
Note that we will say that {\em $F_\alpha$ exhibits no  phase transition} if and
only if the $\alpha$-Farey free energy   function $v$ is differentiable
everywhere, that is,
the right and left derivatives of $v$ coincide everywhere.
\begin{thm}[Lyapunov spectrum for $\alpha$-Farey systems]\label{multi2}
Let $\alpha$ be a  partition that is
either expanding, or expansive and eventually decreasing.
The Hausdorff dimension function of the Lyapunov spectrum associated with
$F_{\alpha}$
is then given as follows.
For $s_-:=\inf \{-(\log a_{n})/n: n \in \N\}$ and
$s_+:=\sup\{-(\log a_{n})/n: n \in \N\}$, we have
that $\sigma_{\alpha}\left(s\right)$ vanishes outside
the interval $[s_-, s_+]$ and for each $s\in (s_-, s_+)$, we have
\[
\sigma_{\alpha}\left(s\right)=
\inf_{u\in\R}\left(u+ s^{-1} v\left(u\right)\right).\]
Concerning the possibility of phase transitions for $F_{\alpha}$, the following
hold:
\begin{itemize}
    \item If
$\alpha$ is expanding, then $F_{\alpha}$ exhibits no phase
transition.  In particular, $v$ is strictly decreasing and  bijective.
\item  If $\alpha$ is expansive of exponent
$\theta$ and eventually decreasing,
then $F_{\alpha}$ exhibits no phase
transition if and only if $\alpha$  is of infinite type.
In particular, $v$ is non-negative and vanishes on $[1,\infty)$.
\end{itemize}
\end{thm}

The structure of the paper is as follows. In Section \ref{prelimaries}, we will
collect various basic properties of the $\alpha$-Farey map
and the $\alpha$-L\"{u}roth map. In particular, this will include a
discussion of the topological dynamics of these two maps and the way in which
they give rise to a
family of distribution functions which are all in the spirit of the Minkowski
question mark function (see \cite{min},\cite{Salem} and \cite{mink?}).
Then, we will locate the invariant densities associated with the
$\alpha$-Farey system and the $\alpha$-L\"{u}roth system
and also establish exactness for both of these maps.

In Section \ref{sec:renewal}, we  study the sequence of Lebesgue measures
of the $\alpha$-sum-level sets, defined above. We first show that this
sequence satisfies a renewal-type equation. We then employ the
discrete
Renewal Theorem by Erd{\H o}s, Pollard and Feller (\cite{EPF}), as well
some renewal results by Garsia, Lamperti (\cite{GL}) and
Erickson
(\cite{Erickson}), and show how these give rise to the proof of
 Theorem \ref{renewal}.

In Section \ref{sec:multi}, we will give a complete description of the multifractal
spectra arising from the $\alpha$-Farey map and the
$\alpha$-L\"{u}roth map. For this we use a general method obtained in  \cite{JaerischKess09}.
 Furthermore, we give a detailed discussion of the phenomenon of phase transition.
 These are the main steps in the proofs of Theorem \ref{multi1} and Theorem \ref{multi2}.

In the Appendix, we will first consider the map $F_{\alpha_H}$,
arising from the {harmonic partition}
$\alpha_H$. As already mentioned above, the associated map
$L_{\alpha_H}$ coincides
with the alternating
L\"{u}roth map.
 We end the paper by giving various further examples which demonstrate the diversity
 of different behaviours of the
spectra given by Theorems \ref{multi1} and \ref{multi2} in dependence
on the chosen partition $\alpha$.

\proc{Remark.} Let us briefly comment also on  the behaviour of the Lyapunov spectra at their boundary points.
Note that in all the examples  given at the end of the paper (see Figures \ref{fig:PressureSpecCL},
\ref{fig:PressureSpectrum3}, \ref{fig:PressureSpectrumPhaseTransition},
\ref{fig:PressureSpectrumNoPhaseTransition}, \ref{fig:PressureSpectrum4_5} and
 \ref{fig:PressureSpectrumTent3_2}) we have that $\sigma_{\alpha}(s_{+})= \tau_{\alpha}(t_{-})=0$.
However, in general  this is not necessarily  true. For instance, one immediately verifies that for a partition $\alpha$ for which $a_{1}=a_{2}$, one
has that $\tau_{\alpha}(t_{-}) \geq (\log 2) /(- \log a_{1}) >0$. Likewise, if  $\alpha$ is given such that
$a_{1}=\sqrt{a_{2}}$,  then $\sigma_{\alpha}(s_{+}) \geq (\log((1+\sqrt{5})/2))/(- \log a_{1})>0$.
Also note that, if $\alpha$ is a partition  which  is expanding and eventually decreasing, then we always have that  $s_{-}>0$, whereas
 $\sigma_{\alpha} (s_{-})$ can be either equal to zero or strictly positive. Furthermore,
 for an expansive partition $\alpha$ we always have that  $s_{-}=0$ and  $\sigma_{\alpha}(0)=1 $.
In order to see that $\sigma_{\alpha}(0)=1 $ is in fact true for any partition  $\alpha$, one  argues as follows.
 On the one hand,  if $\alpha$ is of infinite type,  then this follows from the fact that  $\Lambda(F_{\alpha},x) =0$,
  for $\lambda$-almost all $x \in \mathcal{U}$. On the other hand, if $\alpha$ is of finite type, then the proof follows
  along the lines of the proof of  \cite[Proposition 10]{GR}.
\medbreak
\proc{Remark 2.}
Note that  for the Farey map and its jump transformation,  the Gauss map,
the analogue of Theorem \ref{renewal} has been obtained by the first and the third author in
 \cite{KesseboehmerStratmann:09}. There the results were derived by
using advanced infinite ergodic theory, rather than the strong
renewal  theorems  employed in this paper. This underlines the fact that one
of the main ingredients of infinite ergodic theory is provided by  some delicate estimates in renewal theory. Likewise, as already mentioned above,
the Lyapunov spectra for the Farey map and the Gauss map have been investigated  in detail in  \cite{KesseboehmerStratmann:07}.
The results there are parallel to the outcomes of Theorem \ref{multi1} and \ref{multi2}. Clearly, the Farey map and
the Gauss map are non-linear, whereas the systems in this paper are always
piecewise linear. However,  since our analysis is based on a large family of different partitions of $\mathcal{U}$,
 the class of maps which we consider in this paper
 allows to detect a variety of interesting new
phenomena. For instance, as shown in  \cite{KesseboehmerStratmann:07},  the  spectral sets of  the Farey map
and the Gauss map intersect at the single point $2 \log ((\sqrt{5}+1)/2)$. The same type of behaviour can also be found
in our piecewise linear setting, as shown in  Fig. \ref{fig:PressureSpectrum4_5}
for $a_n:=\zeta\left(5/4\right)^{-1}n^{-5/4}$,
where $\zeta$ denotes the  Riemann zeta-function.  However, this situation is by no means canonical, as the  harmonic
partition  $\alpha_H$ already shows, where the intersection of the
two spectral sets is equal to the interval $[\log2, (\log 6)/2]$
(cf. Fig. \ref{fig:PressureSpecCL}). The situation can be even more  dramatic, as shown in
Fig. \ref{fig:PressureSpectrumTent3_2} for the partition   $\alpha$
determined by $a_{n}:= 2 \cdot 3^{-n}$. For this partition, the spectral set  associated with the $\alpha$-Farey map is fully contained in the spectral set  of
the $\alpha$-L\"uroth map.
A similar picture arises when one considers the possibility of the existence of  phase transitions.
The results of  \cite{KesseboehmerStratmann:07} clearly show that neither the Gauss map nor the Farey map exhibit the type of  phase transition
established in this paper. In contrast to this, Theorem \ref{multi1}
and
\ref{multi2}  show that in the piecewise linear scenario
the situation is much more interesting, as the  examples in Fig. \ref{fig:PressureSpectrum3},
\ref{fig:PressureSpectrumPhaseTransition} and
\ref{fig:PressureSpectrumNoPhaseTransition} clearly demonstrate.
More specifically, if $\lim_{r\searrow t_{\infty}}\sum_{n=1}^{\infty}a_{n}^{r}\log a_{n}/\sum_{n=1}^{\infty}a_{n}^{r}=\infty$
then the dimension function $\tau_{\alpha}$ is real-analytic on $\left(t_{-},\infty\right)$.
An example for this is provided by the alternating L\"uroth system, where $a_{n}=\left(n\left(n+1\right)\right)^{-1}$,
and hence,  $t_{\infty}=1/2$ and  $\sum_{n=1}^{\infty}a_{n}^{t_\infty}=\infty$ (cf. Fig. \ref{fig:PressureSpecCL},
see also Fig. \ref{fig:PressureSpectrum3}, \ref{fig:PressureSpectrum4_5}, and \ref{fig:PressureSpectrumTent3_2} for further examples).
Note that the example considered in Fig. \ref{fig:PressureSpectrumNoPhaseTransition}
is particularly interesting, since it shows that it is possible that there is no phase transition, although $p(t_\infty)$ is finite.
That is, for  $a_{n}:=(n\left(\log
n\right)^{2})^{-2}/\sum_{k=1}^{\infty} (k\left(\log k\right)^{2})^{-2} $,
we have on the one hand
$\sum_{n=1}^{\infty}a_{n}^{t_{\infty}}<\infty$ with  $t_{\infty}=1/2$,
but on the other hand we have
$\lim_{t\searrow t_{\infty}}\sum_{n=1}^{\infty}a_{n}^{t}\log a_{n}/\sum_{n=1}^{\infty}a_{n}^{t}=\infty$.
However, for a partition $\alpha$ for which  $t_{0}:=\lim_{t\searrow t_{\infty}}\sum_{n=1}^{\infty}a_{n}^{t}\log a_{n}/
\sum_{n=1}^{\infty}a_{n}^{t}<\infty$, the $\alpha$-L\"uroth map
$L_{\alpha}$
 exhibits a phase transition of the first kind  at
 $t_{\infty}$. In this case the Hausdorff dimension function
 $\tau_{\alpha}$  is real-analytic on $\left(t_{-},t_{0}\right)$,
 whereas for $t\in[t_{0},+\infty)$ it is explicitly given
by \[\tau_\alpha (t)=\frac{\sum_{n=1}^\infty a_n^{t_\infty}}{t}+t_\infty.\]
An example demonstrating the latter situation is given  in Fig.  \ref{fig:PressureSpectrumPhaseTransition}.
\medbreak

\section{Preliminary discussion of $F_{\alpha}$ and $L_{\alpha}$}\label{prelimaries}
Throughout this section we let $\alpha$ denote some
arbitrary partition  of $\mathcal{U}$ of the type specified at the beginning
of the introduction.
\subsection{Topological properties of $F_{\alpha}$ and $L_{\alpha}$}
Recall from the introduction that
the $\alpha$-L\"uroth map $L_\alpha$  is given by
\[
L_{\alpha}(x):=
\left\{
	\begin{array}{ll}
	    ({t_n-x})/a_n, & \text{ if }x\in A_n,\ n\in\N,\\
	  0, & \hbox{ if } x=0.
	\end{array}
      \right.
\]
In the same way as the Gauss map gives rise to the continued fraction
expansion, the map $L_{\alpha}$ gives rise to a series expansion of
numbers in
the unit interval, which we refer to as the
\textit{$\alpha$-L\"{u}roth expansion}.
More precisely,
let $x\in\mathcal{U}\setminus \{0\}$ be given and let the finite or infinite sequence
$(\ell_k)_{k\geq1}$ of positive integers be determined by $L_{\alpha}^{k-1}(x) \in
A_{\ell_{k}}$, where the sequence terminates in $k$ if and only if  $L_{\alpha}^{k-1}(x)=t_n$, for some $n\geq2$.
Then the $\alpha$-L\"{u}roth expansion of  $x$ is
given as follows, where the sum is supposed to be finite if the
sequence is finite.
\[
x=
\sum_{n=1}^\infty(-1)^{n-1}\left(\textstyle\prod\limits_{i<n}a_{\ell_i}\right)
t_{\ell_n}=t_{\ell_1}-a_{\ell_1}t_{\ell_2}+a_{\ell_1}a_{\ell_2}t_{\ell_3}+\cdots
\]
 In this situation we then write
$x=:[ \ell_1, \ell_2,
\ell_3, \ldots]_{\alpha}$.
It is easy to see
that every infinite expansion is unique, whereas each
$x\in(0,1)$ with a finite $\alpha$-L\"uroth expansion can be expanded
in exactly two ways. Namely, one immediately verifies that $x=[\ell_1, \ldots, \ell_k,
1]_\alpha=[\ell_1,  \ldots, \ell_{k-1}, (\ell_k +1)]_\alpha$. Note that the map $L_\alpha$ only provides the latter expression. By
analogy with continued fractions, for which a number is rational if
and only if it has a finite continued fraction expansion, we say that
$x\in\U$ is an \textit{$\alpha$-rational number} when $x$ has a
finite $\alpha$-L\"uroth expansion and say that $x$ is an \textit{$\alpha$-irrational
number} otherwise. Of course, the set of $\alpha$-rationals is a
countable set.

If we truncate the
$\alpha$-L\"{u}roth expansion of $x$ after $k$ entries we obtain the
{\em $k$-th convergent of $x$}, denoted  $r_k^{(\alpha)}(x)$, which is given by
\[
r_k^{(\alpha)}(x):=[ \ell_1, \ldots,
\ell_k]_{\alpha}=t_{\ell_1}-a_{\ell_1}t_{\ell_2}+\cdots
+(-1)^{k-1}\left(\prod_{i=1}^{k-1}a_{\ell_i} \right)t_{\ell_k}.
\]
Note that if $x=[ \ell_1, \ell_2, \ell_3,
\ldots]_{\alpha}$, then
$L_\alpha(x)=[ \ell_2, \ell_3, \ell_4, \ldots]_{\alpha}$. This shows
that,  topologically, $L_{\alpha}$ corresponds to the shift map on the
space $\N^{\N}$, at least for those points with an infinite $\alpha$-L\"{u}roth expansion.
The cylinder sets associated with  the $\alpha$-L\"{u}roth expansion
are denoted by
\[
C_\alpha (\ell_1, \ldots, \ell_k):=\{[ x_1, x_2, \ldots
]_\alpha:x_i=\ell_i\text{ for } i=1, \ldots,k\}.
\]
We remark here that these cylinder sets are closed intervals with endpoints
given by $[\ell_1, \ldots, \ell_k]_\alpha$ and
$[\ell_1,
\ldots,\ell_{k-1}, (\ell_k +1)]_\alpha$. Consequently, we have  for the Lebesgue
measure  of $C_\alpha
(\ell_1, \ldots,
\ell_k)$ that
\[\lambda(C_\alpha (\ell_1, \ldots, \ell_k))=
\prod_{i=1}^{k}a_{\ell_i}.\]

For the first lemma of this section,  recall that the {\em jump transformation} $F_\alpha^*:\U\to\U$ of
$F_\alpha$ is given by $F_\alpha^*(x)= F_\alpha^{\rho_\alpha(x)}(x)$, where $\rho_\alpha: \U\to \N$
is given by $\rho_\alpha(x):=\inf\{ n\geq 0:F_\alpha^n(x)\in A_1\}+1$.
Note that one can immediately verify that $\rho_\alpha(x)$ is finite for all $x\in\U\setminus \{0\}$.
\begin{lem}\label{lem:2.1}
The jump transformation $F_\alpha^*$ of the $\alpha$-Farey map  $F_\alpha$ coincides with the
$\alpha$-L\"{u}roth map $L_\alpha$.
\end{lem}
\begin{proof}
    First note that if $x=[1, \ell_2, \ell_3, \ldots]_\alpha\in A_1$, for some $\ell_{2}, \ell_{3},\ldots \in \N $,
then $\rho_\alpha(x)$  is clearly equal to $1$. Thus, $F_\alpha^{*}(x)=F_\alpha(x)$, which is
equal to $L_\alpha(x)$, since $L_\alpha|_{A_1}=F_\alpha|_{A_1}$.

Secondly, for $n\geq 2$ we have that $x\in A_n$ if and only if $x=[n, \ell_2, \ell_3,
\ldots]_\alpha$, for some $\ell_{2}, \ell_{3},\ldots \in \N $.   We
then have that
\[
F_\alpha^{*}(x)=F_\alpha^{n}([n, \ell_2, \ell_3,
\ldots]_\alpha)=F_\alpha^{n-1}([n-1, \ell_2, \ell_3,
\ldots]_\alpha)=\cdots = [\ell_2, \ell_3, \ldots]_\alpha= L_\alpha(x).
\]
\end{proof}

Let us now describe a Markov partition $\alpha^*$ and its associated coding for the map
$F_\alpha$.  The partition $\alpha^*$
is equal to $\{A, B\}$, where $A:=A_1$ and $B:=\U\setminus A_1$. Each
 $x\in\U$ has an infinite
Markov coding
$x=\langle x_{1}, x_{2},\ldots \rangle_{\alpha} \in \{0,1\}^{\N}$,
which, for each positive integer $k$,  is given by $x_k=1$ if and only if $F_{\alpha}^{k-1}(x)\in A$.
This coding will be referred to
 as the \textit{$\alpha$-Farey coding}.
 The associated cylinder sets
are denoted by
    \[\widehat{C}_{\alpha}(
    x_{1},\ldots,x_{n}):=\{\langle
 y_{1},y_{2}, \ldots \rangle_{\alpha}: y_{k}=x_{k}, \text{ for }k=1,\ldots,n\}.\]
Notice that all of the $\alpha$-L\"uroth cylinder sets are also
$\alpha$-Farey cylinder sets, whereas the converse of this is not true. More precisely,  a given
 $\alpha$-Farey cylinder set
$\widehat{C}_\alpha(0^{\ell_{1}-1}10^{\ell_{2}-1}10^{\ell_{3}-1}
\cdots 0^{\ell_{k}-1}1)$  coincides with the
$\alpha$-L\"uroth cylinder set
$C_\alpha(\ell_1, \ldots, \ell_k)$. Moreover, if the coding of an
$\alpha$-Farey cylinder set ends  in a $0$, then it
 cannot be  represented by a single $\alpha$-L\"uroth cylinder set.

In the sequel, we require the inverse
branches $F_{\alpha,0}$ and $F_{\alpha,1}$
of the map $F_{\alpha}$. With the convention that
$F_{\alpha,0}(0)=0$, it is straightforward to calculate that these are
given by $F_{\alpha,1}(x):=1-{a_1}x$  for $x\in {\mathcal U}$  and
\[F_{\alpha,0}(x):=\frac{a_{n+1}}{a_{n}}(x-t_{n+1})+t_{n+2}\; \text{ for
 }\;  x\in A_{n}, n\in \N.
\]
In preparation for the next lemma, we now describe the $\alpha$-Farey decomposition of the interval
$\U$, which is obtained by  iterating  the maps $F_{\alpha,
0}$ and $F_{\alpha, 1}$ on $\U$. The first iteration gives rise to the
partition $\{\widehat{C}_\alpha(0), \widehat{C}_\alpha(1)\}$. Iterating  a
second time yields the refined partition $\{\widehat{C}_\alpha(00),
\widehat{C}_\alpha(01), \widehat{C}_\alpha(11),
\widehat{C}_\alpha(10)\}$. Continuing the iteration further, we obtain successively
refined partitions of $\U$
consisting of $2^k$ $\alpha$-Farey cylinder sets of the form
$\widehat{C}_\alpha(x_1, \ldots, x_k)$, for every $k \in \N$. It is clear that
exactly half
of these are also $\alpha$-L\"uroth cylinder sets. The endpoints of each of
these so-obtained intervals are $\alpha$-rational numbers,
and every $\alpha$-rational number is obtained in this way.
Finally, note that if $x=[\ell_{1},\ell_{2},\ldots]_{\alpha}$, then
\[
F_\alpha(x):=\left\{
				 \begin{array}{ll}
				  [\ell_1-1, \ell_2, \ell_3,
				  \ldots]_\alpha, & \hbox{for $\ell_1\geq2$;} \\

				  [\ell_2, \ell_3, \ldots]_\alpha,&\hbox{for $\ell_1=1$.}
				 \end{array}
			       \right.
\]
Also observe that if we consider the {\em
dyadic partition}
$\alpha_D$ given by $a_{n}:=2^{-n}$, then
the map  $F_{\alpha_D}$ arising from this particular partition turns out to coincide with  the tent map, given by
\[
F_{\alpha_D} (x):=\left\{
		\begin{array}{ll}
		  2x & \hbox{for $x\in [0,1/2)$;} \\
		  2-2x & \hbox{for $x \in [1/2, 1]$.}
		\end{array}
	      \right.
\]
Before stating the lemma, we remind the reader that
the measure of maximal entropy $\mu_{\alpha}$ for the system $F_\alpha$
is the measure that assigns mass $2^{-n}$ to each $n$-th level
$\alpha$-Farey cylinder set.
\begin{lem}
The dynamical systems $(\mathcal{U}, {F}_{\alpha})$ and
$(\mathcal{U}, F_{\alpha_D})$ are topologically conjugate and
the conjugating
homeomorphism is given, for each $x=[ \ell_1, \ell_2, \ldots]_{\alpha}$, by
\[
{\theta_{\alpha}}(x):=-2\sum_{k=1}^\infty(-1)^k2^{-\sum_{i=1}^k \ell_i}.
\]
Moreover, the map $\theta_{\alpha}$ is equal to the distribution function of
the measure of maximal entropy  $\mu_{\alpha}$  for the $\alpha$-Farey map.
\end{lem}
\begin{proof}
We will first show by induction that the map
$\theta_{\alpha}$ is indeed equal to the distribution function
$\Delta_{\mu_{\alpha}}$ of the measure $\mu_{\alpha}$. To start,
observe that $\Delta_{\mu_{\alpha}}([1]_\alpha)=
1=\theta_\alpha([1]_\alpha)$ and notice that
for each $k\geq2$ the $\alpha$-rational
number $[k]_\alpha$ appears for the first time
in the $(k-1)$-th level of the $\alpha$-Farey
decomposition, as the right endpoint of the cylinder set
$\widehat{C}_\alpha(0, \ldots, 0)$, with code
consisting of $k-1$ zeros. By the definition of the
measure of maximal entropy, we have that
$\Delta_{\mu_{\alpha}}([k]_{\alpha})=2^{-(k-1)}=
{\theta_{\alpha}}([k]_\alpha)$.

Now, suppose that $\Delta_{\mu_{\alpha}}([\ell_1,
\ell_2, \ldots, \ell_k]_\alpha)=\theta_\alpha([\ell_1,
\ell_2, \ldots, \ell_k]_\alpha)$ for every $k$-tuple of
positive integers $\ell_1, \ldots, \ell_k$ and each $1\leq k \leq n$,
for some $n\in\N$. Further suppose that $n$ is even.
(The case where  $n$ is odd proceeds similarly.) We then
have that the points $[\ell_1,\ell_2, \ldots, \ell_n]_\alpha$ and $[\ell_1,\ell_2, \ldots, \ell_n,1]_\alpha$ are, respectively, the left and
the right endpoints of the $\left(\sum_{i=1}^n \ell_i\right)$-th level
$\alpha$-Farey cylinder set
$\widehat{C}_\alpha(0^{\ell_{1}-1}10^{\ell_{2}-1}1\cdots
0^{\ell_{n}-1}1)$. Clearly, this cylinder set has $\mu_{\alpha}$-measure
equal to  $2^{-\sum_{i=1}^n \ell_i}$. Similarly, we have that
the interval bounded by
$[\ell_1,\ell_2, \ldots, \ell_n]_\alpha$ and
$[\ell_1,\ell_2, \ldots, \ell_n,2]_\alpha $ is a $\alpha$-Farey
cylinder set of level $\left(\sum_{i=1}^n \ell_i\right)+1$ and as
such, has $\mu_{\alpha}$-measure equal to $2^{-\sum_{i=1}^n \ell_i-1}$. Continuing in this way, we reach the interval bounded by the points $[\ell_1,\ell_2, \ldots, \ell_n]_\alpha$ and $[\ell_1,\ell_2, \ldots, \ell_n,\ell_{n+1}]_\alpha $, which has $\mu_{\alpha}$-measure equal to $2^{-\sum_{i=1}^{n+1} \ell_i+1}$.

Using this, we are now in a position to finish the proof by induction, as follows.
\begin{eqnarray*}
&&\hspace{-2.2cm}\Delta_{\mu_{\alpha}}([\ell_1, \ldots, \ell_n,\ell_{n+1}]_\alpha)\\
&=&
\Delta_{\mu_{\alpha}}([\ell_1, \ldots, \ell_n]_\alpha)+
\mu_{\alpha}(([\ell_1 \ldots, \ell_n]_\alpha,[\ell_1,\ldots, \ell_n,\ell_{n+1}]_\alpha))
\\&=&\theta_\alpha([\ell_1, \ldots, \ell_n]_\alpha)+2^{-\sum_{i=1}^{n+1} \ell_i+1}=
\theta_\alpha([\ell_1, \ldots, \ell_n, \ell_{n+1}]_\alpha).
\end{eqnarray*}

It remains to show that the map $\theta_\alpha$ is the conjugating homeomorphism from $F_\alpha$ to the tent system. For this, suppose first that $x=[\ell_1, \ell_2, \ldots]_\alpha\in\U\setminus A_1$. Then, $\theta_\alpha(x)$ is an element of $[0,1/2]$ and we have that
\begin{eqnarray*}
F_{\alpha_D}\left(\theta_\alpha(x)\right)&=&2\left(-2\sum_{k=1}^\infty(-1)^k2^{-\sum_{i=1}^k \ell_i}\right)=-2\left(\sum_{k=1}^\infty(-1)^k2^{-(\ell_1-1)-\sum_{i=2}^k \ell_i}\right)\\
&=&\theta_\alpha\left([\ell_1-1, \ell_2, \ell_3,\ldots]_\alpha\right)=\theta_\alpha(F_\alpha(x)).
\end{eqnarray*}
Now, suppose that $x\in A_1$, that is, $x=[1, \ell_2, \ell_3, \ldots]_\alpha$. Then, it follows that $\theta_\alpha(x)\in[1/2,1]$ and we have that
\begin{eqnarray*}
F_{\alpha_D}(\theta_\alpha(x))&=&2-2\left(2\cdot2^{-1}-2\sum_{k=2}^\infty(-1)^k2^{-1-\sum_{i=2}^k \ell_i}\right)\\
&=&  -2\left(\sum_{k=2}^\infty(-1)^{k}2^{\sum_{i=2}^k \ell_i}\right)
=\theta_\alpha\left([\ell_2, \ell_3, \ldots]_\alpha\right)=\theta_\alpha(F_\alpha(x)).
\end{eqnarray*}

\end{proof}
Our next aim is to determine the H\"older exponent and the
sub-H\"older exponent of the map $\theta_\alpha$, for an arbitrary partition $\alpha$.
For this, we define $\kappa(n):=-n\log 2/(\log a_n)$ and set
\[
\kappa_+:=
\inf \left\{\kappa(n) :n\in\N\right\} \hbox{ and }
\kappa_-:=
\sup \left\{\kappa(n) :n\in\N\right\}.
\]
Note that for $\kappa\in(0,\infty)$ a map $S:\U \to \U$ is called {\em $\kappa$-sub-H\"older
continuous} if there exists a constant $c>0$ such that $|S(x)-S(y)|
\geq c |x-y|^{\kappa}$, for all $x,y \in \U$.
\begin{lem}
We have that the map
$\theta_\alpha$ is $\kappa_{+}$-H\"older continuous and
$\kappa_{-}$-sub-H\"older continuous.
\end{lem}

\begin{proof}
In order to calculate the H\"older exponent of $\theta_\alpha$, first
note that
\[
|\theta_\alpha(C_{\alpha}(\ell_1, \ell_2, \ldots, \ell_k))|=
2^{-\sum_{j=1}^k\ell_j}.
\]
This can be seen by simply calculating the image of the endpoints
of this cylinder, or by noting that every $\alpha$-L\"uroth cylinder
$C_{\alpha}(\ell_1, \ell_2, \ldots, \ell_k)$ is an $n$-th level
$\alpha$-Farey cylinder, where $\sum_{j=1}^k\ell_j=n$.
For the same reason, we have that $\mu_\alpha(C_{\alpha}
(\ell_1, \ell_2, \ldots, \ell_k))=|\theta_\alpha(C_{\alpha}
(\ell_1, \ell_2, \ldots, \ell_k))|$, where $\mu_\alpha$ again
denotes the measure of maximal entropy associated to the
map $F_\alpha$.  Suppose first that $\kappa_{+}$ is non-zero.
In that case, we have,
\begin{eqnarray*}
\lambda(C_{\alpha}(\ell_1, \ell_2, \ldots, \ell_k))&=&
\prod_{i=1}^k a_{\ell_i}=\prod_{i=1}^k2^{-\ell_i/\kappa(\ell_i)}\geq \left(\prod_{i=1}^k2^{-\ell_i}\right)^{1/\kappa_{+}}\\
&=&\left(2^{-\sum_{i=1}^k\ell_i}\right)^{1/\kappa_{+}}=|\theta_\alpha(C_{\alpha}(\ell_1, \ell_2, \ldots, \ell_k))|^{1/\kappa_{+}}.
\end{eqnarray*}
Or, in other words,
\[
|\theta_\alpha(C_{\alpha}(\ell_1, \ell_2, \ldots, \ell_k))|\leq \lambda(C_{\alpha}(\ell_1, \ell_2, \ldots, \ell_k))^{\kappa_{+}}.
\]
Now, let $x$ and $y$ be some arbitrary $\alpha$-irrational numbers
in $\mathcal{U}$. There must be a first time during the backwards
iteration of $\mathcal{U}$ under the inverse branches of
$F_\alpha$ in which an $\alpha$-Farey cylinder set appears
between the numbers $x$ and $y$.  Say that this cylinder set appears in the $p$-th stage of
the $\alpha$-Farey decomposition.  If we go on iterating one more time, it
is clear that there are two $(p+1)$-th level
$\alpha$-Farey intervals fully contained in the interval $(x, y)$; moreover, one
of these also has to be an $\alpha$-L\"uroth cylinder set.
Let this $\alpha$-L\"uroth cylinder set be denoted by
$C_{\alpha}(\ell_1, \ell_2, \ldots, \ell_k)$, where
$\sum_{j=1}^k\ell_j=p+1$.
This leads to the observation that, as $C_{\alpha}(\ell_1, \ell_2, \ldots, \ell_k)$ is contained in $(x, y)$,
\[
|x-y|^{\kappa_{+}}>\lambda(C_{\alpha}(\ell_1, \ell_2, \ldots, \ell_k))^{\kappa_{+}}\geq |\theta_\alpha(C_{\alpha}(\ell_1, \ell_2, \ldots, \ell_k))|=2^{-\sum_{j=1}^k\ell_j}.
\]

Consider the interval $(x, y)$ again. It  is contained inside two neighbouring $(p-1)$-th level $\alpha$-Farey intervals, and so
\[
|\theta_\alpha(x)-\theta_\alpha(y)|< 2^{-(p-1)}+2^{-(p-1)}=2^{-(p-2)} =8\cdot2^{-(p+1)}.
\]
Combining these observations, we obtain that
\[
|\theta_\alpha(x)-\theta_\alpha(y)|\leq 8|x-y|^{\kappa_{+}}.
\]

In case $\kappa_{+}$ is equal to zero,  we have that there exists $m\in\N$ with the property that
\[
\kappa(m)=\frac{m\log 2}{- \log a_m}<\frac{1}{q},
\]
that is,
\[
a_{m}<\e^{- mq\log 2}.
\]
So we have that the sequence of partition elements are eventually
exponentially decaying,
and hence, the H\"older exponent of the map $\theta_\alpha$
is necessarily equal to zero.

The proof of the $\kappa_{-}$-sub-H\"older continuity  of
$\theta_{\alpha}$  follows by similar means and is therefore left to the reader.
\end{proof}

\proc{Remark.}
    Note that the thermodynamical significance of the H\"older and sub-H\"older exponents of
    $\theta_{\alpha}$ is that they provide the extreme points of the
    region $(s_{-},s_{+})$ on which the  Hausdorff dimension
    function $\sigma_{\alpha}$
    of $F_{\alpha}$ is non-zero (see Theorem \ref{multi2}).
    More precisely, we have that
    \[ \kappa_{-}= \frac{\log 2}{s_{-}}  \hbox{ and }
    \kappa_{+}= \frac{\log 2}{s_{+}}, \]
    where $\kappa_-=\infty$ if and only if $s_-=0$.
\medbreak

\subsection{Ergodic theoretic properties of $F_{\alpha}$ and
$L_{\alpha}$}

Let us begin this subsection by showing that $L_{\alpha}$ is an exact transformation and specifying
its invariant measure.
For this the reader might like to recall that
a non-singular transformation
$T$
 of a $\sigma$--finite measure space $\left(\mathcal{U},
 \mathcal{B}, \mu\right)$
is said to be exact if for each $B \in
\bigcap_{n\in \N} T^{-n}\left(\mathcal{B}\right)$
we have that either $\mu(B)$ or $\mu(\U \setminus B)$ vanishes.
\begin{lem}\label{Lexact}
The $\alpha$-L\"uroth map $L_{\alpha}$ is measure preserving and exact with respect to  $\lambda$.
\end{lem}
\begin{proof}
For the proof of $L_{\alpha}$-invariance, let $L_{\alpha,n}$ denote the inverse branch of $L_{\alpha}$
associated with the $n$-th atom $A_{n}$ of $\alpha$. These branches are given by $L_{\alpha,n}(x):=-a_n x
+t_n$, for all $n \in \N$ and $ x \in [0,1)$.
Then, a straightforward calculation shows that for each
 element $B$ of the Borel $\sigma$-algebra
 ${\mathcal B}$ on $\U$,
\[
\lambda(L_{\alpha}^{-1}(B))=\sum_{n\in\N}\lambda(L_{\alpha,n} (B))=
\sum_{n\in\N}a_n \lambda(B)=\lambda(B).
\]
This gives the $L_{\alpha}$-invariance of $\lambda$.

The proof of exactness
is an adaptation of the proof of  Kolmogorov's
zero-one law  for the one-sided
Bernoulli shift (see \cite{Kolmogorov}). To see this,
let
$B \in
\bigcap_{n\in \N} L_{\alpha}^{-n}\left(\mathcal{B}\right)$ be given such  that
 $\lambda(B)>0$. Then, there exists a sequence of Borel sets $(B_{n})_{n\in \N}$ such that
 $B_{n} \in \mathcal{B}$ and
 $B=L_{\alpha}^{-n}B_{n}$, for all $n \in \N$. Note that for
every finite union $\mathcal{C}$ of $L_{\alpha}$-cylinder sets
 we have that
 \[\lambda(B \cap \mathcal{C}) =\lambda(B) \lambda(\mathcal{C}).\]
 Indeed, since $\lambda(B)=\lambda(B_{n})$ for all $n\in \N$,  if $m$ is the maximal length of the cylinder sets in
 $\mathcal{C}$,
 then \[\lambda(\mathcal{C} \cap B) =\lambda(\mathcal{C} \cap
 L_{\alpha}^{-m}B_{m} )=
  \lambda(\mathcal{C}) \lambda( B_{m}) =
  \lambda(\mathcal{C}) \lambda(B).\] From this we deduce that
  \[ \lambda(B \cap C) =\lambda(B) \lambda(C), \hbox{ for all } C \in
  \mathcal{B}.\]
Therefore, by choosing $C= \U \setminus B$, we conclude that
\[ 0=\lambda(B \cap (\U \setminus B)) = \lambda(B) \lambda(\U \setminus
B).\]
This shows that $\lambda(B)=1$, and hence  finishes the proof.
\end{proof}
Since exactness clearly implies ergodicity,
the following list of properties of the system $(\U,
\mathcal{B},L_{\alpha},\lambda)$ is derived from routine ergodic
theoretical arguments, and therefore the proofs are left to the reader.

\noindent For $\lambda$-almost every $x\in\mathcal{U}$, we have that:

\begin{itemize}
    \item $ \displaystyle{
    \lim_{n\to \infty}\frac1n \#\{j\leq n: \ell_j(x)=k\}=a_k.
    }$
    \item $ \displaystyle{
    \lim_{n\to \infty}\frac1n
    \log\left(\prod_{j=1}^{n}\ell_j(x)\right) =\sum_{k=1}^\infty a_k\log k.
    }$
    \item $ \displaystyle{
    \lim_{n\to \infty}\frac1n \sum_{j=1}^{n}\ell_j(x)
    =\sum_{k=1}^\infty t_k.
    }$
    \item $ \displaystyle{
    \lim_{n\to \infty}\frac1n \log\left|x-r_n^{(\alpha)}(x) \right| =\sum_{k=1}^\infty a_k\log a_k.
    }$
    \end{itemize}

We now turn our attention to the ergodic theoretical properties
of the
$\alpha$-Farey system.  The first property to note is that $F_{\alpha}$
is a conservative transformation.  This can be seen, for instance, by
observing that $\bigcup_{n=0}^{\infty} F_{\alpha}^{-n}(A_{1}) = \U \setminus\{0\}$,
and hence,  Maharam's Recurrence
Theorem (\cite[Theorem 1.1.7]{Aaronson:97}) applies, giving that
$F_{\alpha}$ is conservative.

Recall that a $\lambda$-absolutely continuous measure $\nu$ on $\U$ is called
$F_\alpha$-invariant if
$\nu \circ F_\alpha^{-1}=\nu$,  or, equivalently, if
$\mathcal{F}_\alpha\left(\1_{\mathcal{U}}\right)=\1_{\mathcal{U}}$,
where $\mathcal{F}_\alpha:L^{1}\left(\nu\right)\to
L^{1}\left(\nu\right)$ denotes the \textit{transfer operator}
associated with the $\alpha$-Farey system. This
is a positive
linear operator
given by \[
\int_B \mathcal{F}_\alpha(f)\
d\nu=\int_{F_\alpha^{-1}(B)}f\ d\nu,\mbox{
for all }f\in L^{1}\left(\nu\right)\text{ and }B\in\mathcal{B}.\]
Also, note that the \textit{Ruelle operator}
$\mathcal{R}_{\alpha}:L^{1}\left(\nu\right)\to
L^{1}\left(\nu\right)$
for the $\alpha$-Farey system is given by \[
\mathcal{R}_{\alpha} \left(f\right)=\left|{F_{\alpha,0}}'\right|\cdot\left(f\circ
F_{\alpha,0}\right)+\left|{F_{\alpha,1}}'\right|\cdot\left(f\circ F_{\alpha,1}\right),\mbox{ for all
}f\in L^{1}\left(\nu\right).\]
With $\psi:=d \nu / d \lambda$ denoting the density of $\nu$, one immediately verifies that $\mathcal{F}_\alpha$
and $\mathcal{R}_{\alpha}$ are related in the
following way:
\[
\mathcal{F}_{\alpha}(f)=
\frac{1}{ \psi}  \cdot \mathcal{R}_{\alpha}
\left( \psi \cdot
f\right), \mbox{ for all }f\in L^{1}\left(\nu\right).
\]
So, in order to verify that a particular function $\psi$ is a
density  which gives  rise to an invariant measure for the
map $F_\alpha$,
it is sufficient to show that $\psi$ is an
eigenfunction of $\mathcal{R}_{\alpha}$.

\begin{lem}\label{density} Up to multiplication by a  constant, there exists a unique
 $\lambda$-absolutely continuous invariant measure $\nu_\alpha$
 for the system
$(\mathcal{U}, \mathcal{B}, F_{\alpha})$. The density $\phi_{\alpha}$ of $\nu_{\alpha}$
is given, up to multiplication by a  constant,  by
\[
\phi_{\alpha}:=\frac{d\nu_\alpha}{d\lambda}=\sum_{n=1}^{\infty} \frac{t_n}{a_n}  \cdot \1_{A_{n}}.
\]
Moreover,  $\nu_{\alpha}$ is a $\sigma$-finite measure, and we have
that $\nu_{\alpha}$ is an infinite measure
 if and only if
 $\alpha$  is of infinite type.
\end{lem}
\begin{proof} Recall that the inverse branches $F_{\alpha,1}$ and $F_{\alpha,0}$ were defined in Section 2.1 above and
note that a straightforward computation shows that for these we have that
\[ \phi_{\alpha} \circ F_{\alpha,1} = t_{1}/a_{1} \cdot \1_{\U} \hbox{  and  } \phi_{\alpha} \circ F_{\alpha,0} =
\sum_{n=1}^{\infty} t_{n+1}/a_{n+1} \cdot \1_{A_{n}}.\]
Moreover, one immediately verifies that
\[ |F_{\alpha,1}'|= a_{1} \cdot \1_{\U} \hbox{  and  }
|F_{\alpha,0}'|=\sum_{n=1}^{\infty} a_{n+1}/a_{n} \cdot \1_{A_{n}}.\]
Using these two observations, it follows that
\begin{eqnarray*}
\mathcal{R}_{\alpha}(\phi_\alpha)&=&|{F_{\alpha,0}}'|\cdot(\phi_\alpha \circ
F_{\alpha,0}) +|{F_{\alpha,1}}'|\cdot
(\phi_\alpha \circ F_{\alpha,1}) \\
&=& t_{1} \cdot\1_{\U} +
\sum_{n=1}^{\infty} \left( \frac{a_{n+1}}{a_{n}} \frac{t_{n+1}}{a_{n+1}}
\right) \cdot \1_{A_{n}}\\
&=& \sum_{n=1}^{\infty} \left( \frac{t_{n+1}}{a_{n}}  +1
\right) \cdot \1_{A_{n}} = \sum_{n=1}^{\infty} \frac{t_n}{a_{n}}
\cdot \1_{A_{n}}
= \phi_\alpha.
\end{eqnarray*}
This proves all but uniqueness in the first assertion of the lemma.

For the second statement of the lemma, a simple calculation shows that
\[
\nu_{\alpha}\left({\mathcal U}\right)=
\nu_{\alpha}\left(\bigcup_{k=1}^{\infty}A_{k}\right)=
\sum_{k=1}^{\infty}\nu_{\alpha}(A_{k})=
\sum_{k=1}^{\infty}\int_{A_{k}} \phi_\alpha\
d\lambda=\sum_{k=1}^{\infty}\frac{t_k}{a_k}\cdot a_k=
\sum_{k=1}^{\infty}t_k.
\]
Finally, note that the uniqueness of $\nu_{\alpha}$ follows, since, as
we will see in Lemma \ref{F-exact} below, we have that
$F_{\alpha}$ is ergodic. By combining this with the fact that
$F_{\alpha}$ is  conservative, an application of
\cite[Theorem 1.5.6]{Aaronson:97} then gives that $\nu_{\alpha}$ is in fact
unique.
This finishes the proof of the lemma.
\end{proof}
\begin{lem}\label{F-exact}
The $\alpha$-Farey map $F_{\alpha}$ is exact.
\end{lem}
\begin{proof}
    Let  $B_0	\in\bigcap_{n\in \N}F_{\alpha}^{-n}
    \mathcal{B}$ be given such that $\lambda(B_0)>0$.  Since $\nu_{\alpha}$ and $\lambda$
are
    absolutely continuous with respect to each other, it is sufficient to show the exactness
    of $F_{\alpha}$ with respect to $\lambda$.
    Therefore,  the aim is to show that
$\lambda(B_0^{c})=0$.
    For this, first note that, since
    $B_0	\in\bigcap_{n\in \N}F_{\alpha}^{-n}
    \mathcal{B}$,  there exists a sequence $\left(B_{n}\right)_{n\in
\N}$ in $\mathcal{B}$ such that $B_{0}=F_{\alpha}^{-n}B_n$, for all $n\in
\N_{0}$.
    Clearly,  we then have that $B_{k+m}=F_{\alpha}^{k}B_m$, for all $k,m\in
\N_0$.
    Secondly, recalling that since $F_{\alpha}$ is conservative, we have that
    $\rho_{\alpha}$ is finite,
$\lambda$-almost everywhere, where $\rho_{\alpha}(x):=\inf\{n\geq0:F_\alpha^n(x)\in A_1\}+1$.
       Also, define $\rho^{(n)}:=\sum_{k=0}^{n-1}\left(\rho_{\alpha}\circ
    \left(L_{\alpha}^k\right)\right)$.
 Using  the facts that
 $\lambda$  is $L_{\alpha}$--invariant and
Bernoulli with respect to
 ${L_{\alpha}}$,  we obtain  for
 $\lambda$-almost every
    $x=\langle x_{1},x_{2}, \ldots \rangle_{\alpha}
    =[\ell_1,\ell_2,\ldots]_{\alpha}$,
\begin{eqnarray*}
    \lambda\left(B_0|\widehat{C}_{\alpha}(x_{1},\ldots,x_{\rho^{(n)}\left(x\right)})
    \right)
      & = &
      \frac{\lambda\left(F_{\alpha}^{-(\rho^{(n)}\left(x\right))}B_{\rho^{(n)}\left(x
    \right)}\cap \widehat{C}_{\alpha}(x_{1},\ldots,x_{\rho^{(n)}\left(x\right)})
    \right)}{\lambda\left(\widehat{C}_{\alpha}(x_{1},\ldots,x_{\rho^{(n)}\left(x
    \right)})\right)} \\
    & = &
	  \frac{\lambda\left( {L_{\alpha}}^{-n}B_{\rho^{(n)}\left(x\right)}\cap C_{\alpha}( \ell_1,\ldots ,\ell_n)\right)}{ \lambda\left(C_{\alpha}(\ell_1,\ldots ,\ell_n)\right)} \\
    & = &
\frac{\lambda\left({L_{\alpha}}^{-n}B_{\rho^{(n)}
    \left(x\right)}\right) \lambda\left(C_{\alpha}(\ell_1,\ldots ,\ell_n)\right)}
{\lambda \left(C_{\alpha}(\ell_1,\ldots ,\ell_n)\right)}\\
       & = &  \lambda \left(B_{\rho^{(n)}  \left(x\right)}\right).
     \end{eqnarray*}
    Also, by the Martingale Convergence Theorem (\cite{DOOB}), we
have for
    $\lambda$-almost
every $x=\langle x_{1},x_{2}, \ldots \rangle_{\alpha}$, that   \begin{eqnarray*}
   \lim_{n\to \infty}  \, \,   \lambda\left(B_0 |
   \widehat{C}_{\alpha}(
x_{1},\ldots,x_{\rho^{(n)}\left(x\right)})
    \right)=\1_{B_0}(x).
     \end{eqnarray*}
Combining these observations, it follows that $B_0$ coincides up to a set of measure zero with the set $\Omega$, where
 $\Omega$ is defined by \[\Omega:=\{x\in \U: \lim_{n\to \infty}
   \lambda \left(B_{\rho^{(n)}
    \left(x\right)}\right)>0  \}.\]
Since, by  assumption,  $\lambda(B_0)>0$, we now have that
$\lambda(\Omega)>0$. Hence, to finish the proof, we are
left to show that
$\lambda(\Omega)=1$. For this, recall that
$\lambda$ is  $L_{\alpha}$-invariant and ergodic. Thus, it is
sufficient to show
that ${L_{\alpha}}^{-1}\Omega\subset \Omega \mod
\lambda$. In other words,
in order to complete the proof, we are left to show that $\lim_{n\to\infty}
\lambda(B_{\rho^{(n)}({L_{\alpha}}(x)) })>0$ implies that
$\lim_{n\to\infty} \lambda(B_{\rho^{(n)}(x) })>0$.
Since
\[ B_{\rho^{(n+1)}(x)}=B_{\rho(x)+\rho^{(n)}({L_{\alpha}}
(x))}=F_{\alpha}^{\rho(x)}B_{\rho^{(n)}({L_{\alpha}}(x)) },\]
the latter assertion would hold if we establish that  for each
$\epsilon>0$ and $\ell \in \N$
there exists $\kappa>0$ such that for all $C\in \mathcal{B}$ with
$\lambda(C)>\epsilon$ we
have $\lambda(F_{\alpha}^{\ell}C)>\kappa$.
Therefore,  assume that $\lambda(C)>\epsilon$, and let
$\alpha_{\ell}^{*}$
denote the $\ell$-th refinement of the Markov partition $\alpha^*$ for the map $F_{\alpha}$.
Also, one clearly can remove an open neighbourhood of the boundary points of the intervals in $\alpha^*_\ell$
to obtain a closed set $U\subset\mathcal{U}$ such that $\lambda(U)>1-\epsilon/2$.
Since,
there are $ 2^{\ell}$ elements in $\alpha_{\ell}^{*}$, this immediately implies that  $\lambda(C\cap B\cap U)> \epsilon
2^{-\ell-1}$, for some $B\in \alpha_{\ell}^{*}$.
By combining the fact  that $F_{\alpha}^{\ell}:B\to \U$ is
bijective and
the fact that by the choice of $U$ there exists a constant  $c>0$ such that  $
\left( d (\lambda\circ
F_{\alpha}^{\ell})/d \lambda \right)(y) > c$ for all $y \in B\cap U$,
it now follows
that $\lambda(F_{\alpha}^{\ell}C)\geq\lambda(F_{\alpha}^{\ell}(C\cap B\cap U)) > c 2^{-\ell-1}\epsilon$.
Hence, by setting in the above $\kappa:=c 2^{-\ell-1} \epsilon$,
the proof follows.
\end{proof}
We end this section by stating the following
applications of
some general results from infinite ergodic theory to
the system $(\U,
\mathcal{B},F_{\alpha},\nu_{\alpha})$. Note that the first, but only the first,  is also valid for $\alpha$ of finite type.
\begin{itemize}
    \item A consequence of \emph{ Hopf's Ergodic Theorem }
    (\cite{hopf}): \\ For each
    non-negative
    $f \in L^{1}(\lambda)$ with $\int_{\U}f \, d
    \lambda > 0$, we have that
    \[ \lim_{n\to \infty} \sum_{k=0}^{n-1} f(F_{\alpha}^{k} (x))
    =\infty, \hbox{ for $\lambda$-almost every $x \in \U$}.\]
    \item A consequence of \emph{Krengel's Theorem} (\cite{Krengel1}): \\
    If $\alpha$  is of infinite
	type, then we have, for each $\epsilon>0$,
	\[\ \ \ \ \ \ \
 \lim_{n \to \infty} \lambda \left(\left\{x \in \U:\left|1/n
	\sum_{k=0}^{n-1} f(F_{\alpha}^{k}(x) )\right|\geq
	\epsilon \right\}\right) =0, \hbox{ for all } f \in L^{1}(\lambda), f \geq 0.
\]
    \item A consequence of \emph{ Aaronson's
    Theorem} (\cite[Theorem 2.4.2]{Aaronson:97}): \\  If $\alpha$  is
    of infinite
	type, then we have, for each $f \in L^{1}(\lambda)$
	such that $f \geq 0$
	and for each
	sequence$(c_{n})_{n \in \N}$ of positive
	integers,
	 that either
	 \[
\liminf_{n\to \infty} \frac{\sum_{j=0}^{n-1} f(F_{\alpha}^{j}
(x))}{c_{n}} =0,\] or, there exists a subsequence $(c_{n_k})_{k\in \N}$  such
that
\[
 \lim_{k\to \infty} \frac{\sum_{j=0}^{n_{k}-1} f(F_{\alpha}^{j}
(x))}{c_{n_{k}}} =\infty.\]
     \item A consequence of \emph{Lin's Criterion for exactness}
     (\cite{Lin}):\\
     Since $F_{\alpha}$ is exact, we have that if $\alpha$  is of
     infinite type, then
    \[ \lim_{n  \to \infty} \int|{\mathcal F}_{\alpha}^{n}(f)|\ d\nu_\alpha = 0, \mbox{
   for all $ f \in L^{1}(\nu_\alpha)$ such that $ \int f \ d\nu_\alpha=0$}.\]
    \end{itemize}

\section{Renewal theory}\label{sec:renewal}

In this section we  study the sequence
of the Lebesgue measures of the $\alpha$-sum-level sets for a given
 partition $\alpha$.
Recall from the introduction that the $\alpha$-sum-level sets  are
given, for each $n \in \N_{0}$,  by
\[
\mathcal{L}^{(\alpha)}_n:=\left\{x \in C_{\alpha}(\ell_1, \ell_2, \ldots,
\ell_k):\sum_{i=1}^k \ell_i=n, \hbox{ for some } k\in\N\right\},
\]
where, for later convenience, we have set
$\mathcal{L}^{(\alpha)}_0:=\mathcal{U}$.
The first members of this sequence
are as follows:{\footnotesize
\begin{eqnarray*}
& \mathcal{U}\\
& C_\alpha( 1)\\
& C_\alpha (2) \cup
C_\alpha( 1,1 )\\
 & C_\alpha( 3)\cup
C_\alpha( 1,2)\cup C_\alpha( 2,1)\cup
C_\alpha(1,1,1)\\
&C_\alpha( 4)\cup
C_\alpha(3,1)\cup C_\alpha(2,2)\cup
C_\alpha(2,1,1)\cup C_\alpha(1,3)\cup
C_\alpha(1,2,1)\cup C_\alpha(1,1,2)\cup
C_\alpha(1,1,1,1)
\end{eqnarray*}}

In order to obtain
 precise rates for the decay of the Lebesgue measure of the $\alpha$-sum-level
sets $\mathcal{L}^{(\alpha)}_{n}$, we employ some
arguments from renewal theory.
We begin our discussion  with the following
crucial observation, which shows that the sequence
of the Lebesgue measures of the $\alpha$-sum-level sets
satisfies a renewal equation.

\begin{lem}[Renewal Equation]\label{renewaltype}

    For each $n\in\N$, we have that
\[
\lambda(\mathcal{L}^{(\alpha)}_{n})=\sum_{m=1}^{n}a_m\lambda(\mathcal{L}^{(\alpha)}_{n-m}).
\]
\end{lem}

\begin{proof} Since $\lambda(\mathcal{L}^{(\alpha)}_0)= 1$ and
$\lambda(\mathcal{L}^{(\alpha)}_1)= a_{1}$,
  the assertion clearly holds for $n=1$. For $n\geq 2$, the following calculation finishes the proof.
\begin{eqnarray*}\lambda(\mathcal{L}^{(\alpha)}_{n})&=&\lambda(C_{\alpha}(
n ))+\sum_{m=1}^{n-1}\sum_{\genfrac{}{}{0pt}{1}{C_{\alpha}(\ell_1, \ldots,
\ell_{k},
m)\in\mathcal{L}^{(\alpha)}_{n}}{k \in \N}}\lambda(C_{\alpha}(\ell_1,
\ldots, \ell_{k}, m))\\&=&
\lambda(C_{\alpha}(n ))+\sum_{m=1}^{n-1}a_m\sum_{\genfrac{}{}{0pt}{1}{C_{\alpha}(
\ell_1, \ldots,
\ell_{k})\in\mathcal{L}^{(\alpha)}_{n-m}}{k \in \N}}\lambda(C_{\alpha}(
\ell_1, \ldots, \ell_{k}))\\&=&
a_n\lambda(\mathcal{L}^{(\alpha)}_{0})+\sum_{m=1}^{n-1}a_m\lambda(\mathcal{L}^{(\alpha)}_{n-m})=\sum_{m=1}^{n}a_m\lambda(\mathcal{L}^{(\alpha)}_{n-m}).
\end{eqnarray*}
\end{proof}

We are now in the position to give the proof of Theorem \ref{renewal}.

\begin{proof}[Proof of  Theorem \ref{renewal} (1)]
Let us begin with by recalling the statement of the
standard discrete {\em
Renewal Theorem} by Erd{\H o}s, Pollard and Feller (\cite{EPF}).
This theorem considers an infinite probability vector $(v_n)_{n \in
\N}$,
that is,  a sequence of
non-negative real numbers for which $\sum_{k=1}^{\infty} v_{n}=1$.
Associated to this vector,  there exists a sequence $(w_n)_{n\in \N_{0}}$
such that
 $w_0=1$ and such that $(w_{n})$ satisfies the {\em renewal equation} $
w_n=\sum_{m=1}^n v_m w_{n-m}$, for all $n\in\N$. A pair
$((v_{n}),(w_{n}))$ of sequences with these properties will be
referred to as a \emph{renewal pair}.
A simple inductive argument immediately yields that $0 \leq
w_{n}\leq 1$, for all $n \in \N_{0}$. It was shown in \cite{EPF}
that with these hypotheses one then has
that
\[
\lim_{n\to \infty} w_n= \frac{1}{\sum_{m=1}^\infty m\cdot v_m},
\]
where the limit is  equal to zero if the series in the denominator
diverges.

This general form of the discrete renewal theorem  can now
be  applied directly to our specific
situation, namely, the sequence
of the Lebesgue measures of the $\alpha$-sum-level sets.
For this, fix
some partition $\alpha=\{A_{n}:n \in \N\}$, and set
$v_n: = \lambda(A_{n}) =a_n$, for each $n\in\N$. Notice that this is
certainly a probability vector. Then,  put
$w_n:=\lambda(\mathcal{L}^{(\alpha)}_n)$,
for each $n\in \N_{0}$.  In light of Lemma \ref{renewaltype} and the observation that
$w_{0}=\lambda(\mathcal{L}^{(\alpha)}_0)=1$, we then have that
these particular sequences $(v_{n})$ and $(w_{n})$
form indeed a renewal pair.
Consequently, by also observing that $\sum_{k=1}^n k a_{k}\sim
\sum_{k=1}^n t_{k}$, an application of the discrete renewal
theorem immediately implies
that
\[
\lim_{n\to\infty}\lambda(\mathcal{L}^{(\alpha)}_n)=\left( \sum_{k=1}^\infty
t_{k} \right)^{-1},
\]
where this limit is equal to zero if
$\sum_{k=1}^\infty t_{k}$ diverges. Note that, by Lemma \ref{density},
the divergence of the latter series is equivalent to the statement
that  $\alpha$ is of infinite
type.

For the remaining assertion in (1),  let us consider the two generating
functions $a$ and $\ell$, which are given by $a(s):=
\sum_{n=1}^{\infty} a_{n} s^{n}$ and $\ell(s) :=
\sum_{m=0}^{\infty} \lambda(\mathcal{L}^{(\alpha)}_m) s^{m}$.
Using Lemma \ref{renewaltype} and the fact that $\lambda(\mathcal{L}^{(\alpha)}_0)=1$, one  immediately verifies
that for $s \in (0,1)$ we have that $\ell(s)-1=\ell(s) a(s)$, and
hence,  $\ell(s) =1/(1-a(s))$. Since
$a(1)=1$, this gives that $\lim_{s \nearrow 1} \ell(s) =\infty$, which
shows that $\sum_{n=0}^{\infty} \lambda(\mathcal{L}^{(\alpha)}_n)$
diverges.
This finishes the proof of Theorem \ref{renewal} (1).
\end{proof}

\begin{proof}[Proof of  Theorem \ref{renewal} (2)  {(i), (ii)} and
Remark 1]
The statements concerning partitions $\alpha$  of finite type follow easily from part (1).
Similarly to the proof of part (1), the remainder of the proof here follows
from applications of
some general results from renewal theory to the particular
situation of the $\alpha$-sum-level sets.  In order to recall
these results, let $((v_{n})_{n \in \N}, (w_{n})_{n \in \N_{0}})$ be a
given renewal pair, and let the two associated sequences
$(V_{n})_{n\in \N}$ and $(W_{n})_{n\in \N}$ be defined by
$V_{n}:= \sum_{k=n}^{\infty} v_{k}$ and $W_{n}:=\sum_{k=1}^{n}w_{k}$, for
all $n \in \N$. (Note that $\sum_{k=1}^{n}V_{k} \sim
\sum_{k=1}^{n} k v_{k}$.) Let us now first recall the  following
strong renewal theorems obtained by Erickson, Garsia and
Lamperti.  The principle assumption in these results is that
\[V_{n} = \psi(n) n^{-\theta},\]
for all $n \in \N$, for some
 $\theta \in [0,1]$ and for some
 slowly varying function $\psi$.

\noindent {\bf The strong renewal results by Garsia/Lamperti \cite[Lemma 2.3.1]{GL} and
Erickson
 \cite[Theorem 5]{Erickson}}. {\it
 For $\theta \in [0,1]$, we have that
 \[ W_{n}
 \sim (\Gamma(2-\theta) \Gamma(1+\theta))^{-1} \, \cdot \,
 n  \, \cdot  \, \left(\sum_{k=1}^{n} V_{k}\right)^{-1}.\]
 Also, if $ \theta \in (1/2,1]$, then
 \[ w_{n} \sim (\Gamma(2-\theta) \Gamma(\theta))^{-1} \,
 \, \cdot \, \left(\sum_{k=1}^{n} V_{k}\right)^{-1}.\]
 Finally,   for  $\theta \in (0,1/2]$ we have that the limit in the
 latter formula does not have to exist in general. However, in this case we haved that (\cite[Theorem 1.1]{GL})
   \[ \liminf_{n\to \infty}  \,  n \cdot  w_{n} \cdot V_{n} =
   \frac{\sin \pi \theta}{\pi} ,\]
 and that
 if we restrict the index set to the complement of some set of integers of
zero  density, we may replace  the limes inferior by a limit in this equation.
}

 The statements in Theorem \ref{renewal} (2) (i), (ii) and
Remark 1 now follow from  straightforward applications
of these strong renewal results
to the setting of the $\alpha$-sum-level sets, for some given
partition $\alpha$.
For this we  have to put $v_{n}:=a_{n}$, $V_{n}:=t_{n}$ and
$w_{n}:=
\lambda(\mathcal{L}^{(\alpha)}_n)$, and to recall that  the pair $((a_{n})_{n \in
\N}, (\lambda(\mathcal{L}^{(\alpha)}_n))_{n \in \N_{0}})$ satisfies
the conditions of a
renewal pair.
\end{proof}
\proc{Remark.}
Note that, by combining the fact that $\mathcal{L}^{(\alpha)}_{n}
= F_{\alpha}^{-(n-1)} (\mathcal{L}^{(\alpha)}_{1})$ and Lin's criterion for
exactness, as stated at the end of the previous section, one immediately verifies
    that  if $\alpha$ is of infinite type, then
    $\lim_{n\to\infty}\lambda(\mathcal{L}^{(\alpha)}_{n})=0$. Clearly, this
 gives an alternative proof of the first part of Theorem \ref{renewal} (1)
 for the case in which $\alpha$  is of infinite type.
    \medbreak

\section{Multifractal Formalisms for $F_{\alpha}$
and $L_{\alpha}$}\label{sec:multi}

For the proofs of Theorem \ref{multi1} and Theorem \ref{multi2}, we employ the following  general thermodynamical
result obtained by Jaerisch and Kesseb\"ohmer, slightly adapted to fit our particular situation.

\textbf{The general thermodynamical results by Jaerisch and Kesseb\"ohmer (\cite{JaerischKess09})}.
 {\em Let $\alpha$ be given as in the introduction and
consider the two potential functions $\phi,\psi:\mathcal{U}\to\R$
given for $x\in A_n$, $n\in\N$, by $\phi\left(x\right):=\log a_{n}$
and $\psi\left(x\right):=z_{n}$, for some fixed sequence
$(z_n)_{n\in\N}$ of negative real numbers.
For all $s \in \R$ we then have that
\[
\dim_{H}\left\{ x\in\mathcal{U}:\lim_{n\to\infty}\left(\sum_{k=0}^{n-1}
\psi(L_{\alpha}^{k}(x))\right)/\left(\sum_{k=0}^{n-1}
\phi(L_{\alpha}^{k}(x))\right)=s\right\}
\leq \max\{0, -t^{*}\left(-s\right)\}.\]
Here, the function $t:\R\to\R\cup\left\{ \infty\right\} $ is given
by
\[
t\left(v\right):=\inf\left\{
u\in\R:\sum_{n=1}^{\infty}a_{n}^{u}\exp
\left(vz_{n}\right)\leq 1\right\}
\label{eq:FreeEnergy}\]
and $t^{*}$ is the Legendre transform of $t$, that is,  \[
t^{*}\left(r\right):=\sup_{v\in\R}\left(-t\left(v\right)+v r\right).\]
Furthermore,
there exist $r_-, r_+ \in \R$ such that for
$s\in\left(r_{-},r_{+}\right)$, we have
\[
\dim_{H}\left\{ x\in\mathcal{U}:\lim_{n\to\infty}\left(\sum_{k=0}^{n-1}
\psi(L_{\alpha}^{k}(x))\right)/\left(\sum_{k=0}^{n-1}
\phi(L_{\alpha}^{k}(x))\right)=s\right\}
=-t^{*}\left(-s\right). \]
In fact, the boundary points $r_{-}$ and $r_{+}$  are determined explicitly by \[
r_{-}:=\inf\left\{ -t^{+}\left(v\right):v\in\Int\left(\dom
\left(t\right)\right)\right\} \mbox{ and }  r_{+}:=\sup\left\{ -t^{+}\left(v\right):v\in\Int
\left(\dom\left(t\right)\right)\right\} ,\]
where $t^{+}$ denotes the derivative of $t$ from
the right, $\Int\left(A\right)$ denotes the
interior of the set $A$, and $\dom\left(t\right):=\left\{ v\in\R:t\left(v\right)<+\infty\right\} $
refers to the effective domain of $t$. }

\proc{Remark 3.} Note that for $s\in \R $ we have
    \[\left\{ x\in\mathcal{U}:\lim_{n\to\infty}\left(\sum_{k=0}^{n-1}
\psi(L_{\alpha}^{k}(x))/\sum_{k=0}^{n-1}
\phi(L_{\alpha}^{k}(x))\right)=s\right\} \neq \emptyset \]
 if and only if  $\inf\{ z_{n}/\log a_{n}: n \in \N\}
    \leq s \leq \sup\{ z_{n}/\log a_{n}: n \in \N\}.$
By basic properties of the Legendre transform it follows that
    \[ r_{-}  \geq \inf\{ z_{n}/\log a_{n}: n \in \N\} \; \mbox{
    and }  \; r_{+}\leq
    \sup\{ z_{n}/\log a_{n}: n \in \N\}.\]
       \medbreak

In preparation for the proof of Theorems \ref{multi1} and
\ref{multi2}, let  us  also
make the following observation.

\begin{lem}\label{eq:asympExpansive}
    Let $\alpha$ be a partition such that $\lim_{n\to
    \infty} t_{n}/t_{n+1}=\rho\geq1$ and such that $\alpha$ is either
    expanding, or expansive of exponent $\theta$  and eventually decreasing.
    We then have that the following
    hold.
    \begin{enumerate}
    \item We have that\[
    \lim_{n\to\infty}\frac{\log a_{n}}{n}=\lim_{n\to\infty}\frac{\log t_{n}}{n}=-\log \rho.\]
    Furthermore, if $\alpha$  is  expansive of exponent
    $\theta>0$ and eventually decreasing, then we
    have that
    \[
    a_{n}\sim\theta n^{-1} t_n.\]
    \item  If $\alpha$  is expanding or  expansive of exponent
    $\theta>0$ and eventually decreasing, then
    \[\lim_{n\to \infty}\frac{a_{n}}{a_{n+1}}=\rho\]
    \item There exists a sequence $(\epsilon_{k})_{k\in \N}$, with $\lim_{k\to
    \infty} \epsilon_{k}=0$, such that for all $n \in \N$ and $x\in
    \bigcup_{k\geq n} A_{k}$ we have that
    \[ \left|
    \frac{1}{n}\sum_{k=0}^{n-1}\log\left|F_{\alpha}'(F_{\alpha}^{k}(x))\right| - \log \rho \right| < \epsilon_{n}.\]
    \end{enumerate}
\end{lem}
\begin{proof} Let us first prove the assertion in (1).
Since $\lim_{n\to \infty}(\log t_{n}-\log t_{n+1})=\log\rho$, we conclude, by using Ces\`{a}ro averages,
that for  $\rho\geq1$ we have that
\[
\lim_{n \to \infty}\frac{\log t_{n}}{n}=\lim_{n\to \infty}
\frac{1}{n}\left(\log t_{1}+\sum_{k=1}^{n-1} \left(\log t_{k+1}-
\log t_{k}\right)\right)=-\log \rho.
\]
Since $t_{n}-t_{n+1}=a_{n}\leq t_{n}$, this in particular also gives
the first equality  in (1)
for $\rho>1$.
The second statement in (1) follows from the Monotone Density
Theorem (\cite{BinghamGoldieTeugels:89}, Theorem 1.7.2). Clearly,
this in particular also implies that $\lim_{n\to \infty}(\log
a_{n})/n= 0$ for the case $\rho=1$. This completes the proof of the statement in (1).

The proof of (2) for the expansive case is an immediate consequence
of (1),  whereas for the expanding case  the assertion in (2) is an
immediate consequence of the following observation:  \[
\lim_{n\to \infty}\frac{a_{n}}{a_{n+1}}=\lim_{n\to \infty}\frac{t_{n-1}-t_n}{t_n-t_{n+1}}=\frac{\rho-1}{1-1/\rho}=\rho.
\]
For the proof of (3), observe that
\begin{eqnarray*}
    \left|
\frac{\sum_{k=0}^{n-1}\log\left|F_{\alpha}'(F_{\alpha}^{k}(x))\right|}{n} - \log \rho \right|
&\leq & \sup_{k \in \N} \left|\frac{\log a_{k} -\log a_{n+k}}{n} -\log
\rho\right| \\
& =&
\sup_{k \in \N} \left|\frac{\log a_{k}}{k} \cdot  \frac{k}{n} - \frac{\log
a_{n+k}}{n+k}  \cdot \frac{n+k}{n}  -\log
\rho\right| \\
& =:& \epsilon_{n}.\end{eqnarray*}
Since  by (1) we have $\lim_{k\to\infty}(\log  a_{k})/k=-\log\rho$,
it follows that $\lim_{k\to \infty} \epsilon_{k}=0$.
\end{proof}

We are now in the position to prove Theorem \ref{multi1}.

\begin{proof}[Proof of Theorem \ref{multi1}] We apply the general
result by Jaerisch and Kesseb\"ohmer, as stated above, to the special situation in which
$z_{n}:=-1$, for each $n\in\N$. In order to determine the function
$t$, we consider the function $v:(t_\infty, \infty)\to\R$,
which is  given by $v(u):=\log\sum_{n=1}^{\infty}a_{n}^{u}$,
where  $t_{\infty}:=\inf \{r>0: \sum_{k=1}^{\infty} a_{k}^{r} <
\infty\}$.
 On the one hand, if $v\left(t_{\infty}\right)$ is infinite, then
the free energy function $t$  appearing in the result of Jaerisch and
Kesseb\"ohmer is identically equal to the inverse  $v^{-1}$ of $v$.
On the other hand, if $v\left(t_{\infty}\right)=:c<\infty$,
then $t(s)=v^{-1}(s)$ for all $s \in (-\infty,c)$, whereas $t(s)=t_{\infty}$
for all $s\in [c,+\infty)$.
In both cases, one immediately finds, by considering the asymptotic
slopes of $t$,  that
$r_{-}=0$ and $r_{+}= 1/ \inf\{-\log a_{n} : n \in \N\}$.
Hence, using Remark 3 and the general thermodynamical
result stated above, it follows that the boundary
points  of the non-trivial part
of  the
Lyapunov spectrum  associated with the map $L_{\alpha}$  are
determined by $t_{-}:=1/r_{+} =\inf\{-\log a_{n}:n \in \N\}$ and
$t_{+}:= +\infty$
(where the latter follows, since here we have that $r_{-}=0$).

For both of these two cases, this
shows that the Hausdorff dimension function associated with the Lyapunov
spectrum of $L_\alpha$ is  given,
for $s\in(t_{-},+\infty)$, by\begin{eqnarray*}
\tau_{\alpha}(s) & = & -t^{*}\left(-1/s\right)=
\inf_{v\in\R}\left(t\left(v\right)+s^{-1}v\right)=
\inf_{u\in\R}\left(u+s^{-1}\log\sum_{n=1}^{\infty}a_{n}^{u}
\right)
\end{eqnarray*} and $\tau_{\alpha}(s)$ vanishes for $s<t_-$.

For the discussion of the phase transition phenomena for
$L_{\alpha}$, one immediately verifies that for the right derivative of
the pressure function $p$ of $L_{\alpha}$, where the reader
might like to recall  that $p$ is given by
$p(u):=\log\sum_{n=1}^{\infty} a_{n}^{u}$, we have that
\[ p'(u)= \frac{\sum_{n=1}^{\infty} a_{n}^{u} \log
a_{n}}{\sum_{n=1}^{\infty} a_{n}^{u}}.\]
Clearly, $p$ is real-analytic on $(t_{\infty},\infty)$.
Hence, we have that $L_{\alpha}$ exhibits no phase transition
if and only if
$\lim_{u\searrow t_{\infty}} -p'(u) = + \infty$.  We now distinguish
the following two cases.

If $\alpha$ is expanding, then there is no phase transition.
    This follows, since, by Lemma \ref{eq:asympExpansive}, we have that
    $p(u)<\infty$, for all $u>0$. In particular, $t_{\infty}=0$.

    If $\alpha$ is expansive of exponent $\theta>0$  such that $t_{n}=\psi(n)
    n^{-\theta}$,
    then Lemma \ref{eq:asympExpansive} implies that there exists
    $\psi_{0}$ such that $\psi_{0}(n) \sim \theta \psi(n)$ and
    $a_{n}= \psi_{0}(n) n^{-(1+\theta)}$. Consequently, we have that
   $t_{\infty}= 1/(1+\theta)$. Hence, we now observe that
  \[\lim_{u\searrow t_{\infty}} -p'(u) =
(1+\theta)   \lim_{u\searrow t_{\infty}}  \frac{\sum_{n=1}^{\infty}
\left( n^{-(1+\theta)} \psi_{0}(n)\right)^{u} \log \left(n
(\psi_{0}(n))^{-1/(1+\theta)}\right)}{\sum_{n=1}^{\infty}
\left( n^{-(1+\theta)} \psi_{0}(n)\right)^{u} }.
\]For $\theta=0$, this shows that $\lim_{u\searrow t_{\infty}}p'(u)=\infty$ if and only if $-\sum_{n=1}^\infty a_n\log(a_n)=\infty$.
We now split the discussion as follows. Firstly, if
$\sum_{n=1}^{\infty} \psi(n)^{1/(1+\theta)} (\log n)/n$
converges, then, clearly,  in the
latter expression the numerator and the denominator
 both converge, and hence, $\lim_{u\searrow t_{\infty}} -p'(u)$ is
finite, showing that in this case the system exhibits a phase
transition. Secondly, if
$\sum_{n=1}^{\infty} \psi(n)^{1/(1+\theta)} (\log n)/n$
diverges, then we have to consider the following two sub-cases.
If
$\sum_{n=1}^{\infty}
 n^{-1} \psi_{0}(n)^{1/(1+\theta)}$ converges, then $\lim_{u\searrow
 t_{\infty}} -p'(u) = \infty$. On the other hand, if
$\sum_{n=1}^{\infty}
 n^{-1} \psi_{0}(n)^{1/(1+\theta)}$ diverges, then for every $k\in \N$ we have
$(k^{-(1+\theta)} \psi_{0}(k))^{u}/ \sum_{n=1}^{\infty}
( n^{-(1+\theta)} \psi_{0}(n))^{u}\to 0$ as $u\to 1/(1+\theta)$ and hence we have, that
$\lim_{u\searrow
 t_{\infty}} -p'(u) = \infty$.
 Therefore, in both of these sub-cases
 the  system exhibits no phase
transition.

Finally, for the interpretation of $t_\infty$ in terms of the Hausdorff dimension of the Good-type set $G_\infty ^{(\alpha)}$,
we have shown above that $t_\infty=1/(1+\theta)$ for $\alpha$ expansive of exponent $\theta>0$ and $t_\infty=0$ for $\alpha$ expanding. It has been proved in \cite{SM} that for $\alpha$ expansive of exponent $\theta>0$ we have $\dim_H(G_\infty^{(\alpha)})=1/(1+\theta)$. It is clear, by considering coverings of $G_\infty^{(\alpha)}$ by cylinder sets,  that in the case of $\alpha$ expanding, we have that  $\dim_H(G_\infty^{(\alpha)})=0$.
This finishes the proof of Theorem \ref{multi1}.
\end{proof}

In the proof of Theorem \ref{multi2},
 the following proposition will be useful. In this proposition, we consider the
 potential function
$N:\U\to\N\cup\{\infty\}$, which is given by
\[
N(x):=\left\{
        \begin{array}{ll}
          n & \hbox{for $x\in A_n$, for $n\in \N$ ;} \\
          \infty & \hbox{for $x=0$.}
        \end{array}
      \right.
\]

\begin{prop}\label{eq:prop}
Let $\alpha$ be a   partition which is either expanding, or expansive
of exponent $\theta$ and eventually decreasing.
With $$\Pi(L_{\alpha},x):= \lim_{n\to \infty}
\left(\sum_{k=0}^{n-1}\log\left|L_{\alpha}'(L_{\alpha}^{k}(x))\right|\right)/
\left(\sum_{k=0}^{n-1}N(L_{\alpha}^{k}(x))\right), $$
we then have for each $s\geq0$ that the sets
\[ \left\{ x\in\mathcal{U}:\Pi(L_{\alpha},x)=s\right\} \hbox{ and }
\left\{ x\in\mathcal{U}:\Lambda(F_{\alpha},x)=s\right\} \]
 coincide up to a countable set of points. \end{prop}
\begin{proof}
Set $S_{n}\left(x\right):=\sum_{k=0}^{n-1}\log\left|L_{\alpha}'(L_{\alpha}^{k}(x))\right|$,
$T_{n}\left(x\right):=\sum_{k=0}^{n-1}\log\left|F_{\alpha}'(F_{\alpha}^{k}(x))\right|$
and $N_{n}\left(x\right):=\sum_{k=0}^{n-1}N(L_{\alpha}^{k}(x))$. Since
$S_{n}\left(x\right)/N_{n}\left(x\right)$ is a subsequence of $T_{n}\left(x\right)/n$
it follows for all $s\geq 0$ that
\[
\left\{x\in\mathcal{U}:\Lambda(F_{\alpha},x)=s\right\} \subset\left\{
x\in\mathcal{U}:\lim_{n\to \infty}S_{n}\left(x\right)/N_{n}\left(x\right)=s\right\}.
\]
Since the set of preimages of $0$ under $F_{\alpha}$ is at most
countable, we can clearly restrict the discussion to those points $x\in\U$
for which $N\left(F_{\alpha}^{k}\left(x\right)\right)=:\ell_{k}\left(x\right)$
is finite for all $k\in\N$. Now put $k_{n}\left(x\right):=\sup\left\{ k\in\N:N_{k}\left(x\right)\leq n\right\} $
and $m_{n}\left(x\right):=n-N_{k_{n}(x)}\left(x\right)$, and assume
that $\Pi(L_{\alpha},x)=s$, for some $s \geq 0$. Thus,
$\lim_{n \to \infty} S_{k_{n}\left(x\right)}\left(x\right)/N_{k_{n}\left(x\right)}\left(x\right)
=s$, and a  straightforward computation gives that
\begin{eqnarray*}
\frac{T_{n}\left(x\right)}{n} & = &
\frac{S_{k_{n}\left(x\right)}\left(x\right)}{N_{k_{n}\left(x\right)}\left(x\right)+m_{n}
\left(x\right)}+\frac{
T_{m_{n}\left(x\right)}\left(L_{\alpha}^{k_{n}\left(x\right)}
\left(x\right)\right)}{N_{k_{n}\left(x\right)}\left(x\right)+m_{n}
\left(x\right)} \\ &=&
\frac{N_{k_{n}\left(x\right)}(x)}{N_{k_{n}\left(x\right)}\left(x\right)+m_{n}
\left(x\right)}  \cdot \frac{S_{k_{n}\left(x\right)}
\left(x\right)}{N_{k_{n}\left(x\right)}\left(x\right)}  +
\frac{m_{n}
\left(x\right)\left(\log \rho \pm
\epsilon_{m_{n}(x)}\right)}{N_{k_{n}
\left(x\right)}\left(x\right)+m_{n}
\left(x\right)}, \end{eqnarray*}
where $(\epsilon_{k})_{k\in \N}$ denotes the sequence which was obtained in Lemma
\ref{eq:asympExpansive} (3).
For the case  $s=\log \rho$  one immediately verifies, using the
observation that  the latter sum is a convex
combination, that  $\Lambda(F_{\alpha},x)= \log \rho$. Hence, we are left only
to consider the case  $s\neq \log \rho$. Given this assumption,
observe that
\begin{eqnarray*}
\frac{T_{n}\left(x\right)}{n} & = &
\frac{1}{ 1+m_{n}
\left(x\right) /  N_{k_{n}\left(x\right)}(x)}  \cdot \frac{S_{k_{n}\left(x\right)}
\left(x\right)}{N_{k_{n}\left(x\right)}\left(x\right)}  +
\frac{\log \rho \pm
\epsilon_{m_{n}(x)}}{1+N_{k_{n}
\left(x\right)}\left(x\right)/m_{n}
\left(x\right)}.\end{eqnarray*}
Hence, it remains to show that $\lim_{n\to \infty} m_{n}
\left(x\right) /  N_{k_{n}\left(x\right)}(x) =0$. For this we argue by
way of contradiction,
using the inequality $m_{n}(x) \leq \ell_{k_{n}(x)+1} (x)$, as follows.
Put $b_{k_{n}(x)}:=\log
a_{\ell_{k_{n}\left(x\right)}\left(x\right)}$, and  observe that
\begin{eqnarray*}
\lim_{k\to\infty}\frac{S_{k_{n}\left(x\right)+1}\left(x\right)}{N_{k_{n}\left(x\right)+1}\left(x\right)} & = & \lim_{k\to\infty}\frac{S_{k_{n}\left(x\right)}\left(x\right)+b_{k_{n}\left(x\right)+1}}{N_{k_{n}\left(x\right)}\left(x\right)+
\ell_{k_{n}\left(x\right)+1}(x)}\\
& =&  \lim_{k\to\infty}\frac{S_{k_{n}\left(x\right)}\left(x\right)\left(1+\frac{b_{k_{n}\left(x\right)+1}}{S_{k_{n}\left(x\right)}
\left(x\right)}\right)}{N_{k_{n}\left(x\right)}\left(x\right)\left(1+\frac{\ell_{k_{n}\left(x\right)+1}(x)}{N_{k_{n}\left(x\right)}
\left(x\right)}\right)}.\end{eqnarray*}
Now suppose, by way of contradiction, that
$\lim_{n\to\infty}\ell_{k_{n}\left(x\right)+1}\left(x\right)/N_{k_{n}\left(x\right)}\left(x\right)\neq0$.
Since $S_{k_{n}\left(x\right)}\left(x\right)$ is strictly increasing,
we then have that there exists a strictly increasing sequence of positive integers  $(n_j)_{j\in\N}$
such that
$\lim_{j\to\infty}b_{k_{n_j}\left(x\right)+1}\left(x\right)=\infty$.
This implies that
\[
\lim_{j\to \infty}\frac{b_{k_{n_j}\left(x\right)+1}}{\ell_{k_{n_j}\left(x\right)+1}(x)}=\log\rho.\]
 By combining this with the calculation above, we obtain that \[
1=\lim_{j\to\infty}\frac{b_{k_{n_j}\left(x\right)+1}N_{k_{n_j}\left(x\right)}\left(x\right)}
{\ell_{k_{n_j}\left(x\right)+1}(x)S_{k_{n_j}\left(x\right)}\left(x\right)}=\frac{\log\rho}{s}\neq 1,\]
which is a contradiction and hence  finishes the proof.
\end{proof}
\begin{proof}[Proof of Theorem \ref{multi2}]
Let us begin by showing that $-s_{-}$ and $-s_{+}$ are  the
asymptotic slopes of the $\alpha$-Farey free energy function $v$. This
follows, since for each $\epsilon>0$ and for $u>0$, resp. $u<0$, we
have
\[ \sum_{n=1}^{\infty} \exp\left(n u \left(\frac{\log a_{n}}{n} +
s_{\mp} \mp \epsilon
\right)\right)  \left\{
	\begin{array}
	 {l@{\, \, \, \, \, \hbox{for}\,\, }l}
	\leq  \sum_{n=1}^{\infty} \exp\left(\mp nu\epsilon \right)
	\longrightarrow 0  &  u   \longrightarrow\pm \infty\\  \geq
	\exp\left(\pm u\epsilon \right)  \longrightarrow +\infty &
	u \longrightarrow\pm \infty.
	\end{array} \right.\]
The next step is to examine the
possibility of the existence of phase transitions. For this, we introduce the function $Z$, given by
\[ Z(u,v) := \sum_{n=1}^{\infty} \exp\left(n \left(\frac{u\log a_{n}}{n} -
v\right)\right).\]
Let us again consider the expanding and the expansive case
separately.

If $\alpha$ is expanding, then we immediately have that $Z$ is real-analytic in both
variables $u$ and $v$, and also, that $Z$ is strictly decreasing in $v$.
Moreover, for each $u_{0} \in \R$ fixed, we have that
$\{Z(u_{0},v):v \in \R\} = (0,\infty)$. This implies that for
each $u \in \R$ there exists a unique $f(u) \in \R$ such that
$Z(u,f(u))=1$. An application of the Implicit Function Theorem then
gives that $f$ is real-analytic and coincides with the $\alpha$-Farey
free energy   function $v$.
 It follows that in the expanding case the system exhibits no phase
 transition.

 For $\alpha$ expansive and if $u$ is strictly less than $1$, we
 can argue similarly to the expanding case, which then gives the existence of a
 real-analytic function
 $f:(-\infty,1) \to \R^{+}$ such that $Z(u, f(u))=1$ and such
 that $f(u)=v(u)$, for all $u \in (-\infty,1)$. For $u>1$, one then
 immediately verifies that
 \[\sum_{n=1}^{\infty}  a_{n}^{u} \e^{-wn}
 \left\{
	 \begin{array}
	  {l@{\, \, \, \, \, \hbox{for}\,\, }l}
	<1  &  w\geq 0\\  = \infty
	  &
	  w< 0.
	 \end{array} \right.\]
It follows that $v(u)=0$, for $u\geq 1$, which then shows
that in this case the system exhibits a phase transition if and only if $\lim_{u \nearrow
1}  f'(u) <0$. In order to investigate the expansive situation in greater detail, note
that, using the Implicit Function Theorem again, an elementary
calculation gives that, for $u<1$,  we have that
\[  f'(u)=  \frac{\sum_{n=1}^{\infty}
a_{n}^{u} \e^{-f(u)n} \log a_{n}}{\sum_{n=1}^{\infty}
n a_{n}^{u} \e^{-f(u)n}} .
\]
For the case in which $\alpha$  is of infinite type, we have that
the denominator in the above expression tends to infinity, for $u$
tending to $1$ from below. Indeed, since for each $N\in \N$
and $u<1$, we have
\[  \sum_{n=1}^{\infty}   n \, a_{n}^{u} \e^{-f(u) n} \geq
\e^{-f(u) N} \sum_{n=1}^{N}   n \, a_{n} \to  \sum_{n=1}^{N}   n \, a_{n}
, \mbox{ for } u \nearrow 1.\]
Using the fact that $\lim_{n\to \infty} ((\log
a_{n})/n)=0$, it follows that
\[ \lim_{u \nearrow 1} f'(u)=  \lim_{u \nearrow 1}
\sum_{n=1}^{\infty}  \frac{\log a_{n}}{n}  \cdot
\frac{n a_{n}^{u} \e^{-f(u)n}}{\sum_{k=1}^{\infty}
k a_{k}^{u} \e^{-f(u)k}}  =0.
\]
Summarising these observations, we now have that if $\alpha$  is of infinite type
then the system exhibits no phase transition.

Hence, it only remains to consider the case
 in which $\alpha$  is of finite type. Here, the easiest situation to analyse
 occurs for $\theta >1$. Clearly, in this case we have that in the above
 expression for $f'$ the
 denominator and the numerator  both
converge
to a finite value not equal to zero,  for $u$  tending to $1$ from below. Therefore, in
this case the system exhibits a phase transition.

Finally, it remains to consider the case in which $\alpha$  is of finite type and $\theta =1$.
In fact,  the following argument  requires  only  that  $\sum_{n}na_{n}=\sum_{n}t_{n}<\infty$,
and  hence it will also give an alternative proof  for the case
$\theta > 1$.
For this, let $v_{N}:\R\to\R$ be given by \[
\sum_{n=1}^{N}a_{n}^{u}\e^{-v_{N}\left(u\right)n}=1.\]
It is easy to check that  $v_{N}$ is real-analytic and that it converges
pointwise to the $\alpha$-Farey
free energy function $v$.
Also,  define $\delta_{N}:= \sum_{n>  N} n a_{n} / \sum_{n\leq  N}
n a_{n}$  and observe that $\lim_{N\to \infty} \delta_{N}=0$.
 Using  the fact that $\e^{ax}\geq ax+1 $, for all $a,x\geq0$, we then
 have that
\[ \sum_{n\leq N} a_{n} \e^{\delta_{N} \cdot n/N}\geq
\sum_{n\leq N}a_{n}  +  \frac{\delta_{N}}{N} \sum_{n\leq N} n a_{n}
= \sum_{n\leq N}a_{n}  +  \frac{1}{N} \sum_{n>  N} n a_{n}
\geq \sum_{n=1}^{\infty}a_{n} =1.\]
Combining this with the definition of $v_N$, it follows   that $v_N(1)\geq
- \delta_N/N$
and
hence,
\[
\sum_{n=1}^{N}na_{n}\leq \sum_{n\leq N}na_{n}\e^{-v_{N}
\left(1\right)\cdot n}\leq\e^{-v_{N}\left(1\right)\cdot N}\sum_{n=1}^{\infty}na_{n}\leq \e^{\delta_N}\sum_{n=1}^{\infty}na_{n}.
\]
This gives that $$
\lim_{N \to \infty} v_N'(1)= \lim_{N \to \infty}\frac{\sum_{n}a_{n}\log a_{n} \e^{-v_{N}(1)n}}{
\sum_{n}na_{n}
\e^{-v_{N}(1)n}}=\frac{\sum_{n}a_{n}\log a_{n}}{\sum_{n}na_{n}}<0.$$
Since $v_{N} \leq f$ on $(-\infty, 1)$,  we have that $\lim_{u \nearrow  1}f'(u)\leq
\lim_{N\to \infty}v_{N}'(1)$. Combining these observations,  it  now
follows that
$\sum_{n}a_{n}\log a_{n}/\sum_{n}na_{n}$ is an upper bound
for $\lim_{u \nearrow  1}f'(u)$. The fact that this is also a
 lower bound is an  immediate consequence of  the  following
 calculation.
\[ \lim_{u \nearrow  1}\frac{\sum_{n=1}^{\infty}
a_{n}^{u} \e^{-f(u)n} \log a_{n}}{\sum_{n=1}^{\infty}
n a_{n}^{u} \e^{-f(u)n}} \geq \lim_{u \nearrow  1} \frac{\sum_{n=1}^{\infty}
a_{n}^{u}  \log a_{n}}{\sum_{n=1}^{\infty}
n a_{n} \e^{-f(u)n}} = \frac{\sum_{n}a_{n}\log a_{n}}{\sum_{n}na_{n}}. \]
This shows that also in this case
the system exhibits a phase transition.

In order to derive the description of $\sigma_{\alpha}$ in terms of
the $\alpha$-Farey free energy function, as stated in the theorem,
we  apply the above stated general result by Jaerisch and Kesseb\"ohmer to the special situation in which
 $z_{n}:=-n$, for all $n \in \N$.  This gives the Hausdorff dimension
 function associated with $\left\{ x\in\mathcal{U}:\Pi(L_{\alpha},x)=s\right\} $, which, by
Proposition \ref{eq:prop}, coincides with the Hausdorff dimension function of
the Lyapunov spectrum associated with $F_{\alpha}$.
Let us now distinguish two cases, the first
in which the Farey-system exhibits no phase
transition and the second in which it has a phase transition.
In the first case, the boundary points
of the  spectral set  are given by
$s_-=1/r_+$ and $s_{+}=1/r_{-}$. This can be shown in a  similar fashion to the
$\alpha$-L\"uroth case, by observing that $t$ coincides  with the inverse
$v^{-1}$ of the $\alpha$-Farey free energy function $v$.
More precisely,
for $\alpha$  expanding this holds on $\R$, whereas if $\alpha$ is
expansive then this is true on $[0,\infty)$ (and in this case $t(v)
=\infty$ for all $v \in (-\infty,0)$).
Moreover, if $\alpha$ is expansive and exhibits no phase transition,
then
we have that $s_{-}=0$ and $r_+=\infty$.
 Therefore,
it follows that
$\sigma_{\alpha}\left(s\right)=-t^{*}\left(-1/s\right)$, for all
$s\in(s_{-},s_{+})$. This gives the proof of the first part  of Theorem
\ref{multi2} for the case in which there is no phase transition.

Finally, if there exists a phase transition  then, by the above,
we necessarily have that $\alpha$ is expansive and
$r_{+}=-\left(\sum_{n}na_{n}\right)/\left(\sum_{n}a_{n}\log
a_{n}\right)<\infty$, showing that $0=s_{-}< 1/r_{+}$. By the general result  of
Jaerisch and Kesseb\"ohmer,
  the dimension formula stated in the theorem then  holds for all
$s\in ( 1/r_{+}, s_{+})$. For $s \in (0, 1/r_{+}]$ we have that
$-t^{*}\left(-1/s\right)=1$, which immediately gives the
 upper bound $1$ for
$\sigma_{\alpha}(s)$,  for all $s \in (0,1/r_{+}]$. The fact that $1$
is also the  lower bound
on $(0, 1/r_{+}]$   is an immediate
consequence of \cite{JaerischKess09}, Corollary 1.9 (3) (Exhaustion Principle II) (see also Example 1.13 in  \cite{JaerischKess09}). This finishes the proof of Theorem \ref{multi2}.
\end{proof}

\section{Some Examples}\label{sec:lueroth}

As mentioned already in the introduction, if we choose the harmonic
partition
$\alpha_H$ we
obtain the $\alpha_{H}$-Farey map $F_{\alpha_H}$, which is given explicitly by
\[
F_{\alpha_H}(x)=\left\{
        \begin{array}{ll}
          2-2x, & \hbox{for $x\in A_1$;} \\
          \frac{n+1}{n-1}x-\frac1{n(n-1)}, & \hbox{for $x\in A_n$, $n
	  \geq 2$};\\
	  0 , & \hbox{for $x=0$.}
        \end{array}
      \right.
\]
From the map $F_{\alpha_H}$, by the method of  Lemma \ref{lem:2.1}, we
obtain the alternating L\"{u}roth map $L_{\alpha_H}$. Recall from the introduction that this map is given by
\[
L_{\alpha_H}(x)=\left\{
        \begin{array}{ll}
          -n(n+1)x+(n+1), &  \hbox{for $x\in A_n$, $n \in \N$;}\\
	 0 , & \hbox{for $x=0$.}
        \end{array}
      \right.
        \]
The corresponding $\alpha_{H}$-L\"uroth expansion of some arbitrary
$x = [\ell_{1},\ell_{2},\ldots]_{\alpha_{H}}\in\mathcal{U}$ is given by
\[
x=\sum_{n=1}^\infty\left((-1)^{n-1}(\ell_n+1)\prod_{k=1}^n(\ell_k(\ell_{k}+1))^{-1}
\right).
\]
 Also, the Lebesgue measure
of the cylinder set  $C_{\alpha}( \ell_1, \ldots, \ell_k)$ is equal to
$1/(\ell_1(\ell_1+1)\cdots \ell_k(\ell_k+1))$.

\begin{figure}[h]\center
\includegraphics[width=0.4\textwidth]{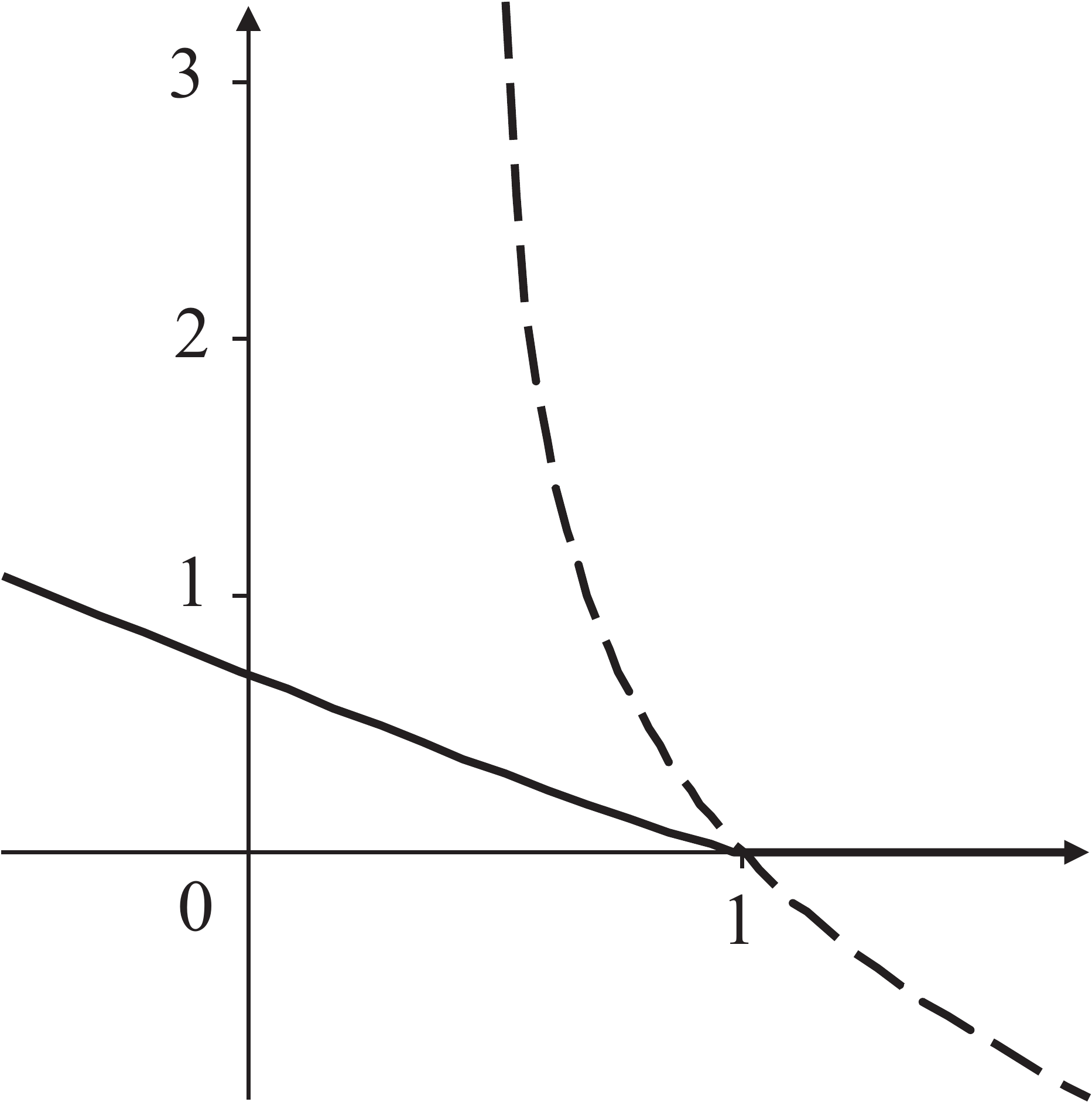}\hspace{0.1\textwidth}
\includegraphics[width=0.4\textwidth]{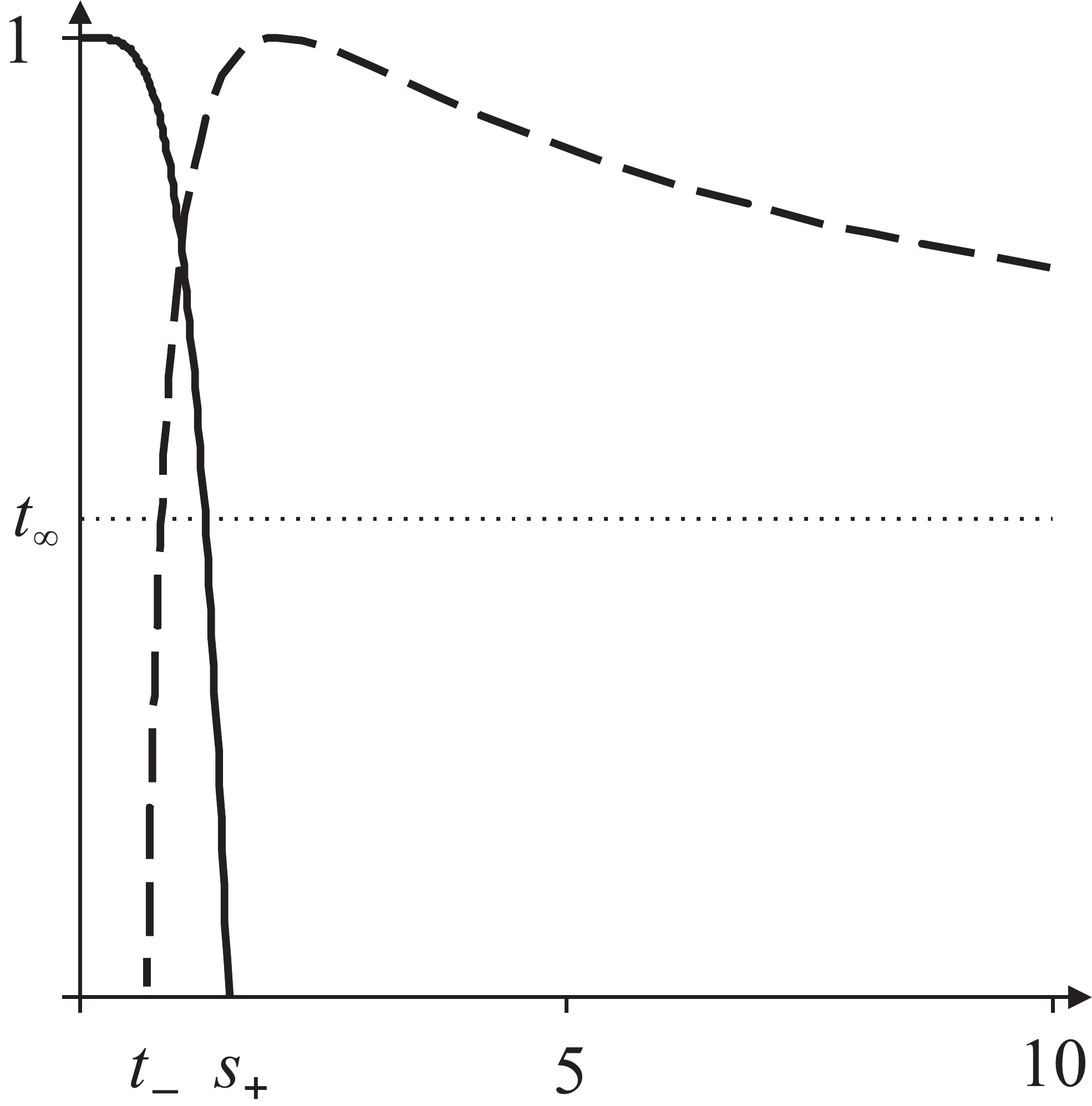}
\caption{{\bf Infinite critical value ${p(t_\infty)=\infty}$ and no phase transition for the $\alpha_H$-Farey
free energy function and the  $\alpha_H$-L\"uroth
pressure function.}  The figure shows the  $\alpha_H$-Farey free energy $v$ (solid line),
the $\alpha_H$-L\"uroth pressure function $p$ (dashed line), and the associated dimension
graphs $\sigma_{\alpha}$ and $\tau_{\alpha}$ of the alternating L\"uroth system. Here,
$t_{-}= \log 2, t_{\infty}=1/2$ and $s_{+} = (\log 6)/2$. Both $F_{\alpha_H}$ and $L_{\alpha_H}$  experience
no phase transition.}
\label{fig:PressureSpecCL}
\end{figure}

 \begin{figure}[h]\center
 \includegraphics[width=0.38\textwidth]{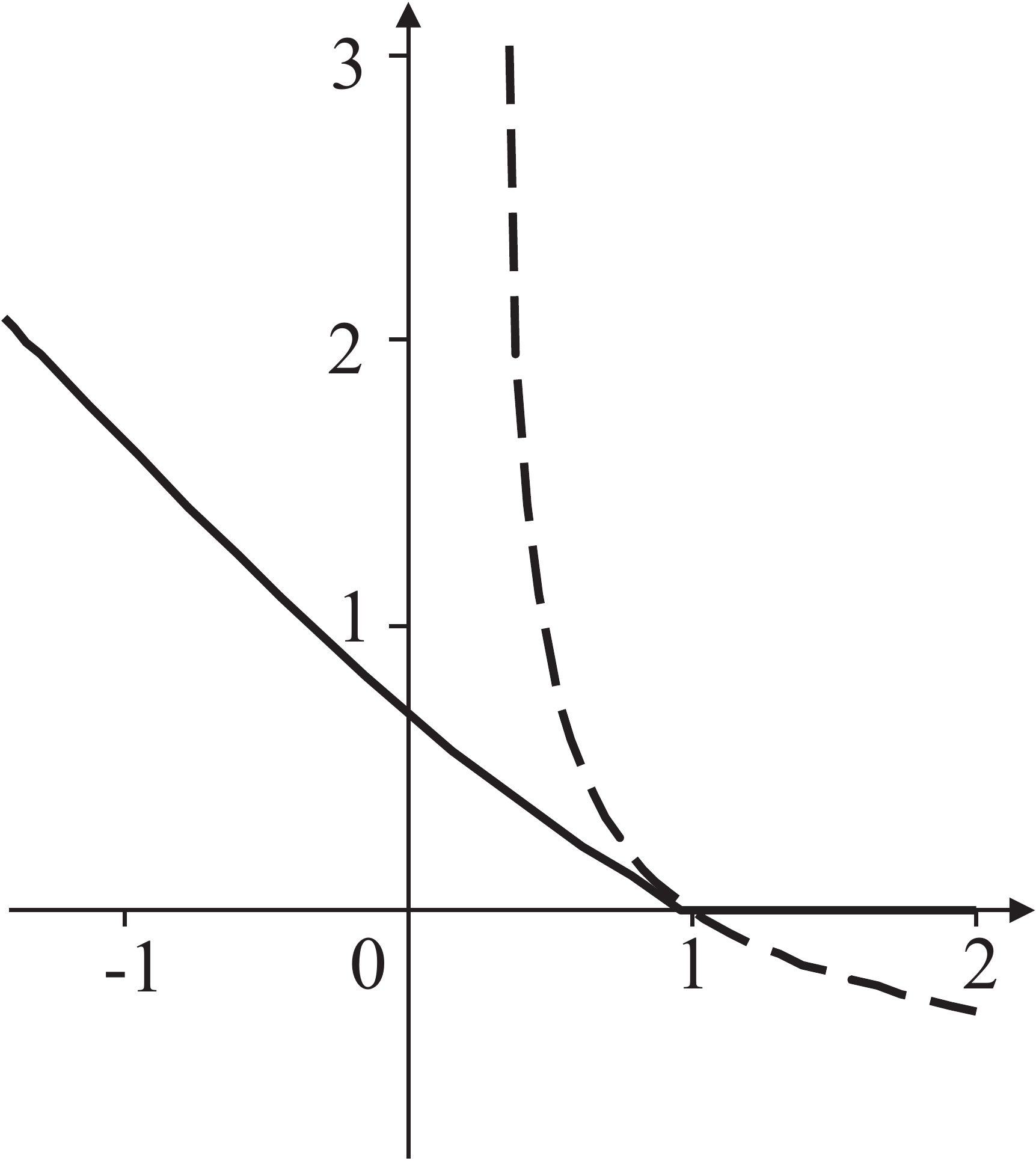}\hspace{0.1\textwidth}
 \includegraphics[width=0.4\textwidth]{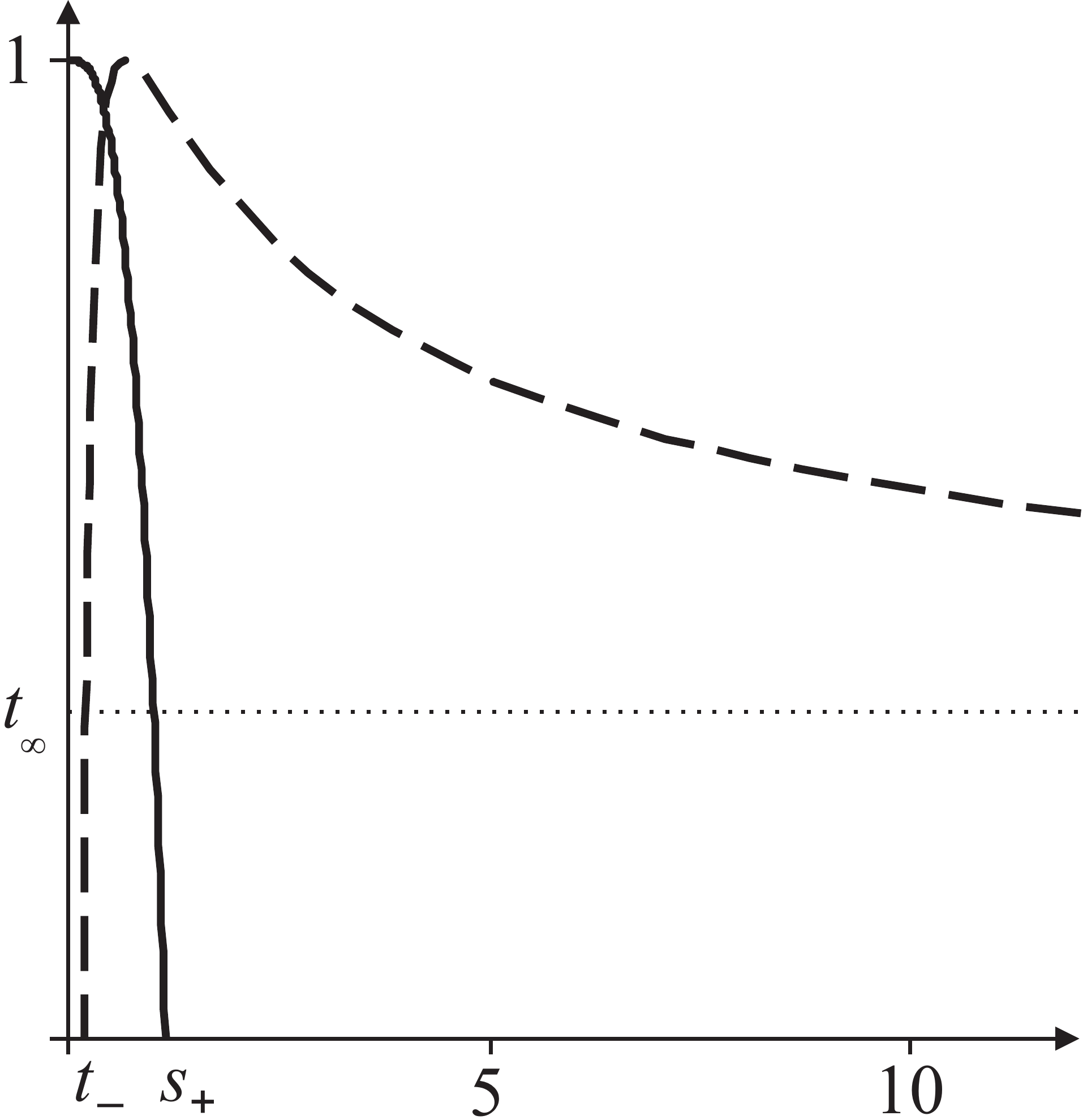}
 \caption{{\bf Phase transition  for the  $\alpha$-Farey free energy function, no phase transition for the
 the $\alpha$-L\"uroth pressure function and $\alpha$ expansive.} The $\alpha$-Farey free energy $v$ (solid line),
 the $\alpha$-L\"uroth pressure function $p$ (dashed line), and the
 associated dimension  graphs  for $a_n:=\zeta\left(3\right)^{-1}n^{-3}$.
 Here,  $F_{\alpha}$ has a phase transition, namely, $p$ is not differentiable at $1$,
 whereas $L_{\alpha}$ exhibits no phase transition and ${p(t_\infty)=\infty}$.}
 \label{fig:PressureSpectrum3}
 \end{figure}

 \begin{figure}[h]\center
  \includegraphics[width=0.4\textwidth]{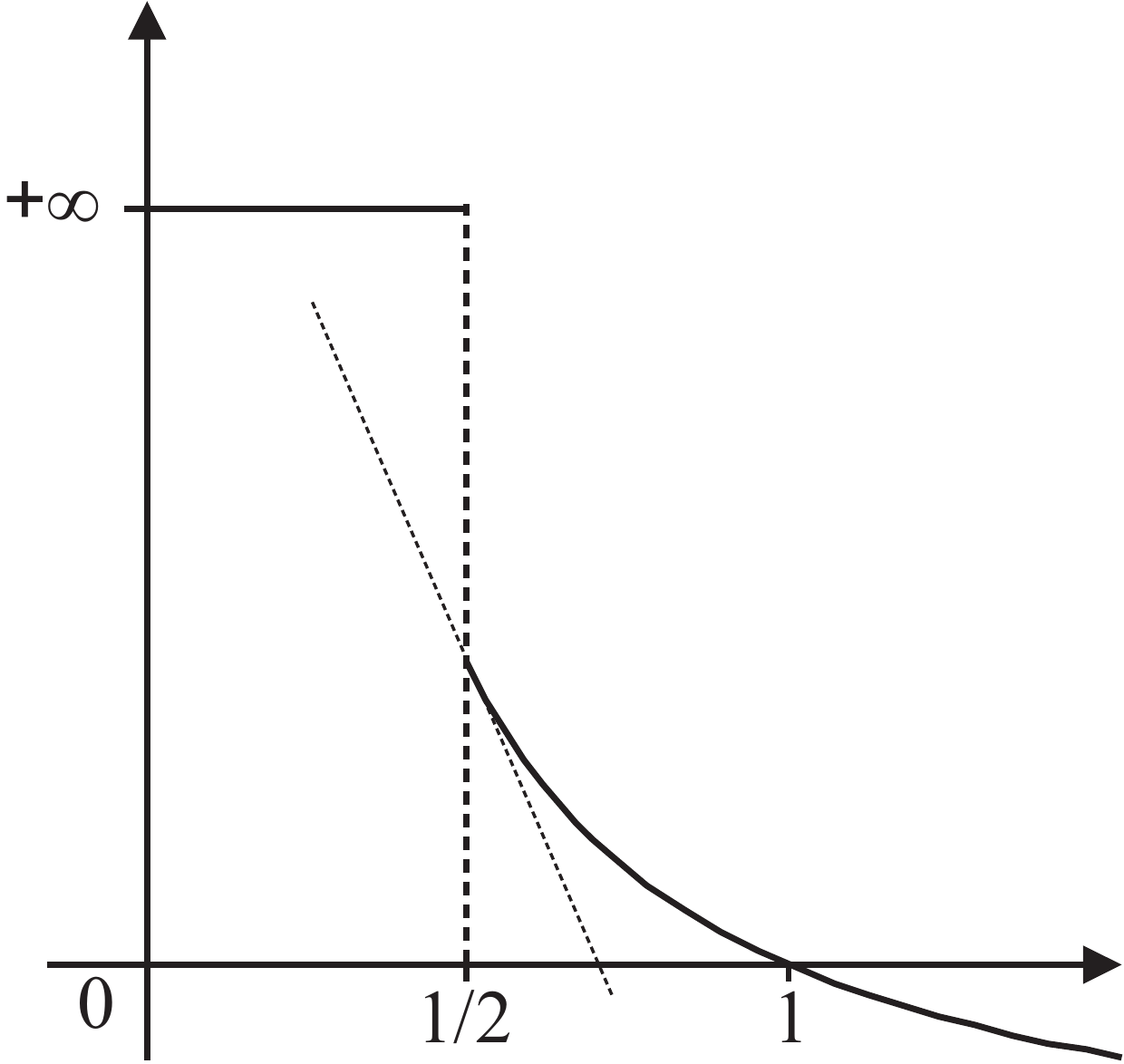}\hspace{0.1\textwidth}
  \includegraphics[width=0.4\textwidth]{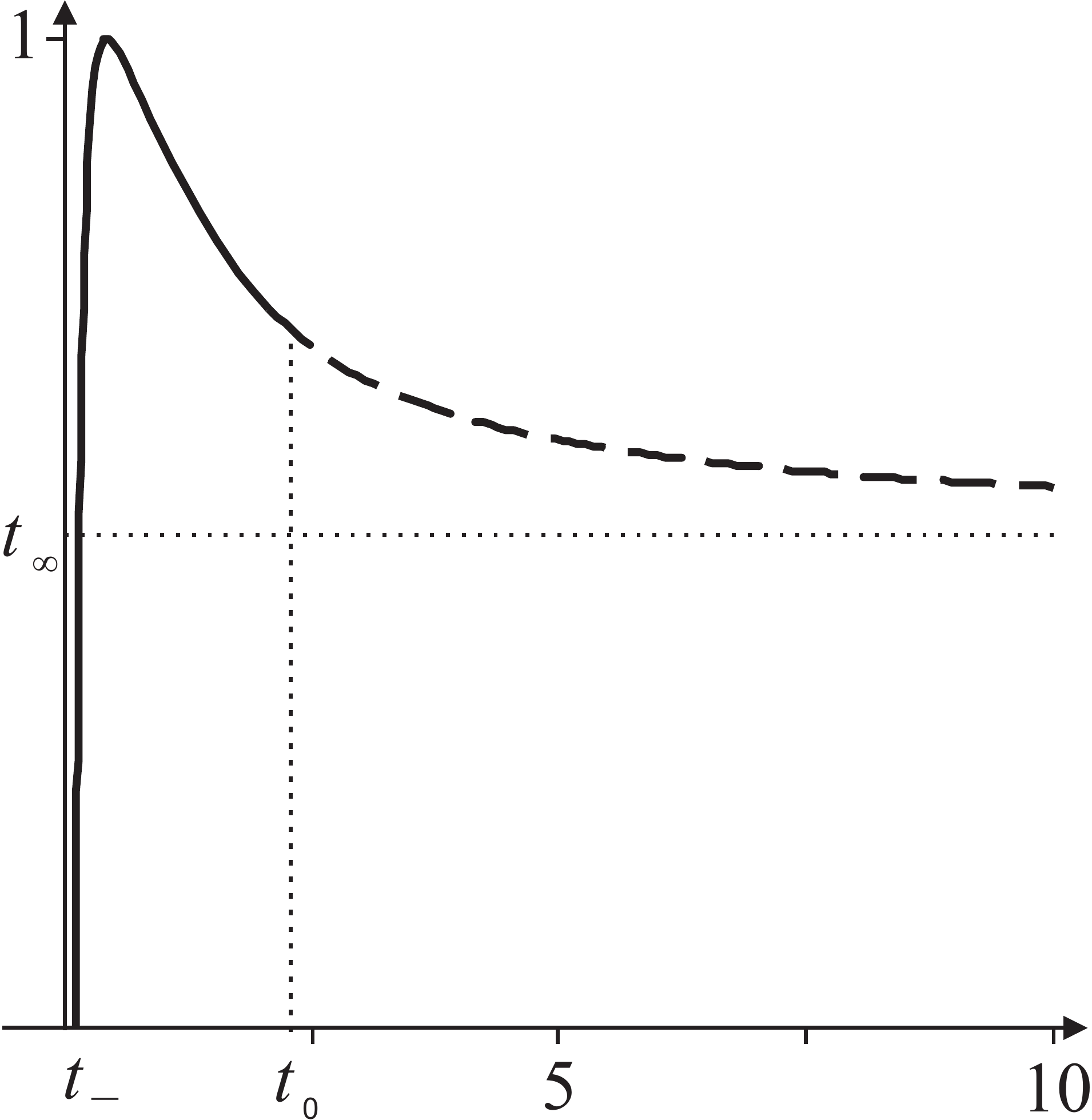}
  \caption{{\bf  Finite critical value ${p(t_\infty)<\infty}$ with phase transition for the $\alpha$-L\"uroth
pressure function and $\alpha$ expansive.}  The  $\alpha$-L\"uroth pressure function $p$, and the
  associated dimension  graphs for the  $\alpha$-L\"uroth system with $a_n:=n^{-2}\cdot (\log(n+5))^{-12}/C$, where $C:=\sum_{n\geq1} n^{-2}\cdot (\log(n+5))^{-12}$.
  In this case $t_\infty =1/2$ and $p(1/2)<\infty $ and $L_{\alpha}$ has a phase transition, namely, $v$ is not differentiable at $1/2$.}
  \label{fig:PressureSpectrumPhaseTransition}
  \end{figure}
 \begin{figure}[h]\center
  \includegraphics[width=0.34\textwidth]{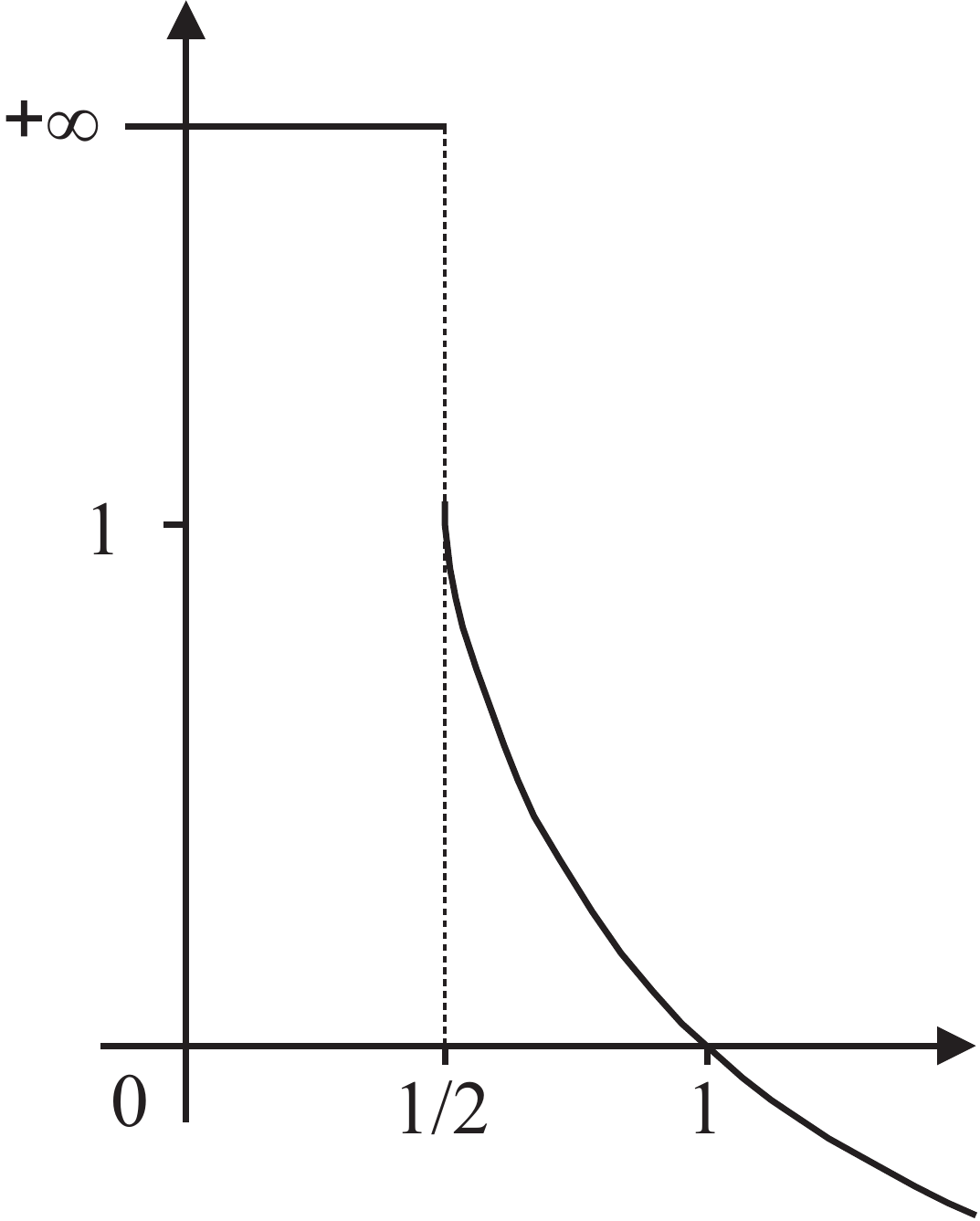}\hspace{0.1\textwidth}
  \includegraphics[width=0.42\textwidth]{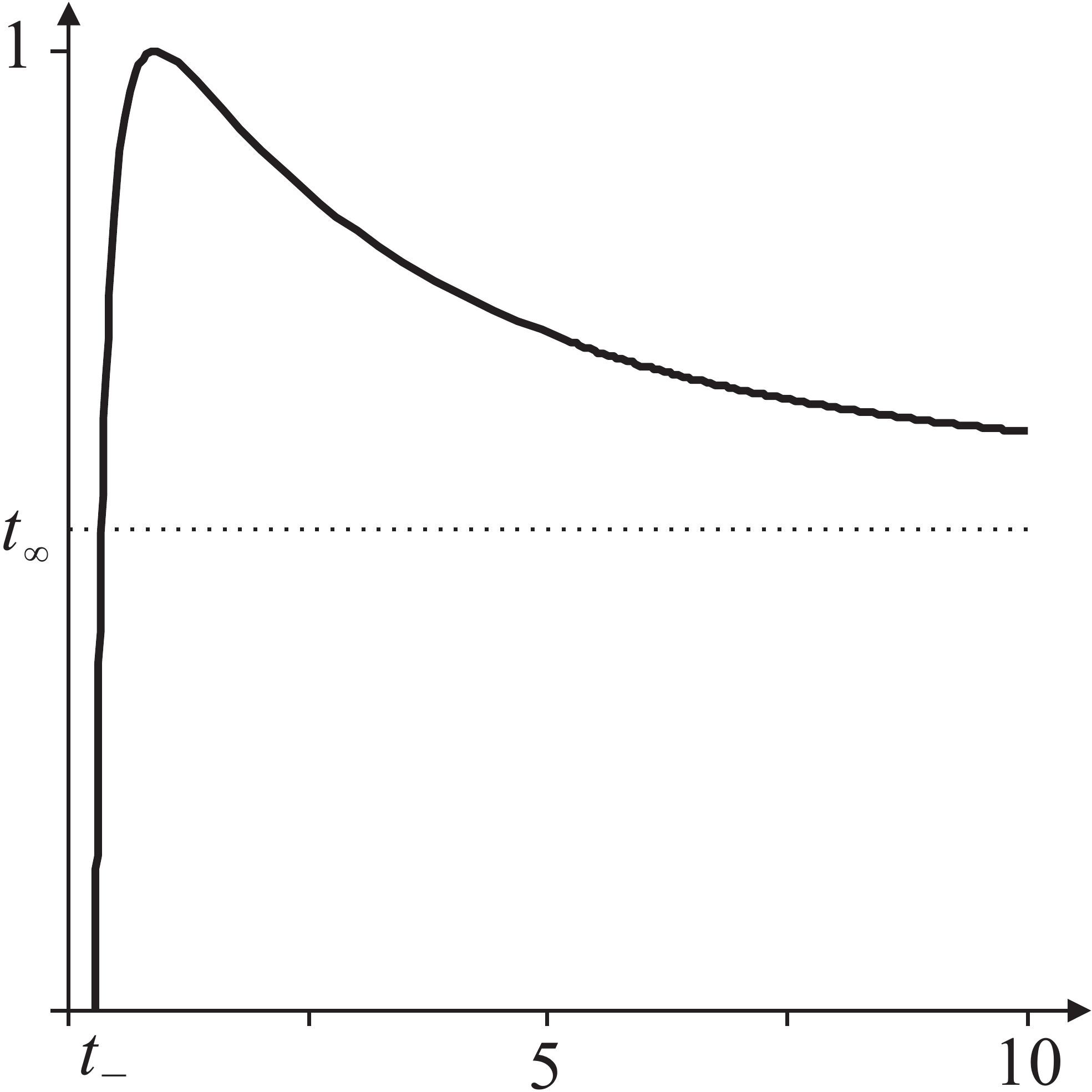}
  \caption{{\bf  Finite critical value ${p(t_\infty)<\infty}$ and no phase transition for the $\alpha$-L\"uroth
pressure function and $\alpha$ expansive.} The  $\alpha$-L\"uroth pressure function $p$, and the
  associated dimension  graphs for the  $\alpha$-L\"uroth system with $a_n:=n^{-2}\cdot (\log(n+5))^{-4}/C$, where $C:=\sum_{n\geq1} n^{-2}\cdot (\log(n+5))^{-4}$.  In this case $t_\infty =1/2$ and $p(1/2)<\infty $,
   but  $L_{\alpha}$ exhibits no phase transition.}
  \label{fig:PressureSpectrumNoPhaseTransition}
  \end{figure}

 \begin{figure}[h]\center
 \includegraphics[width=0.32\textwidth]{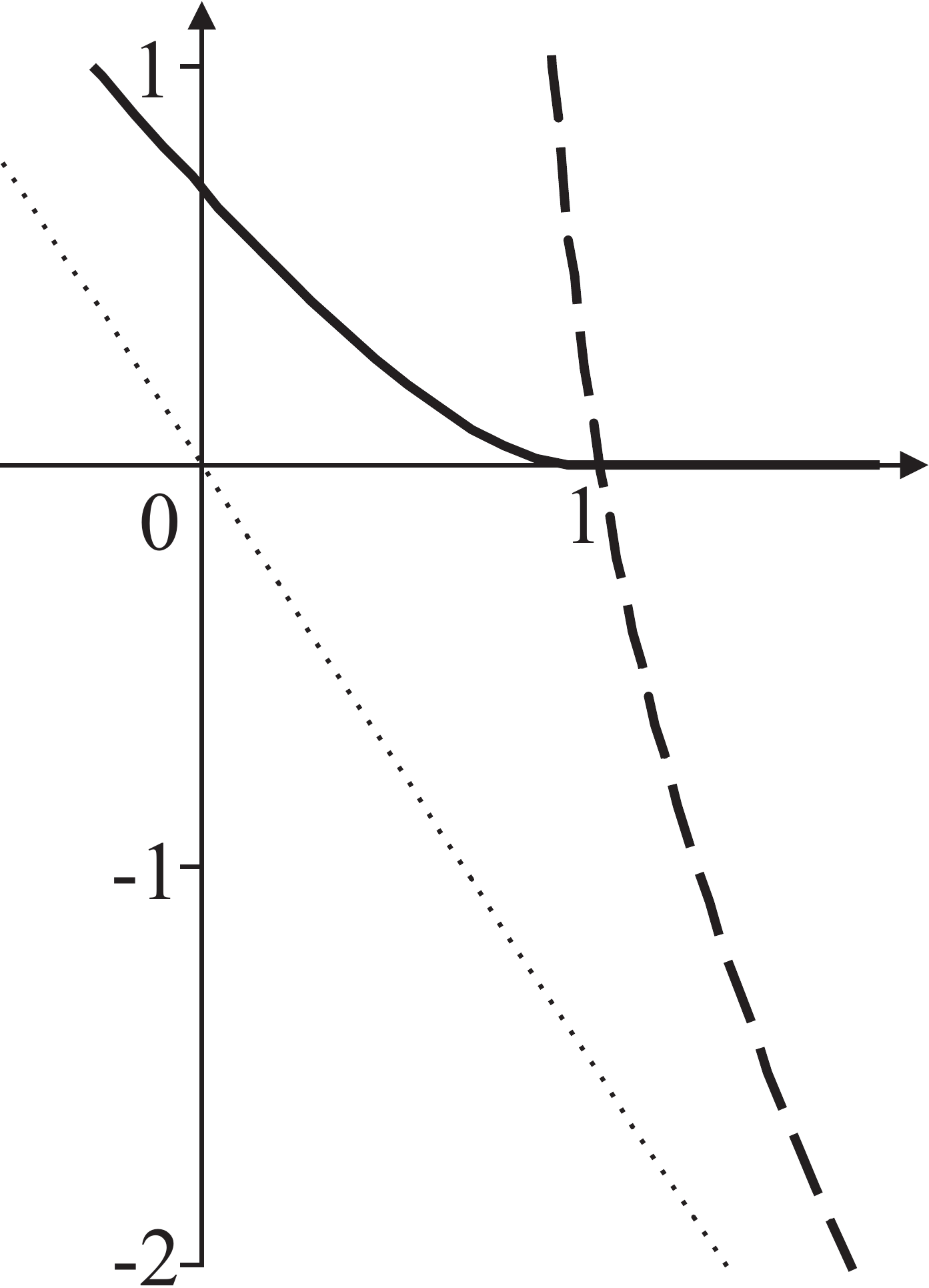}\hspace{0.1\textwidth}
 \includegraphics[width=0.4\textwidth]{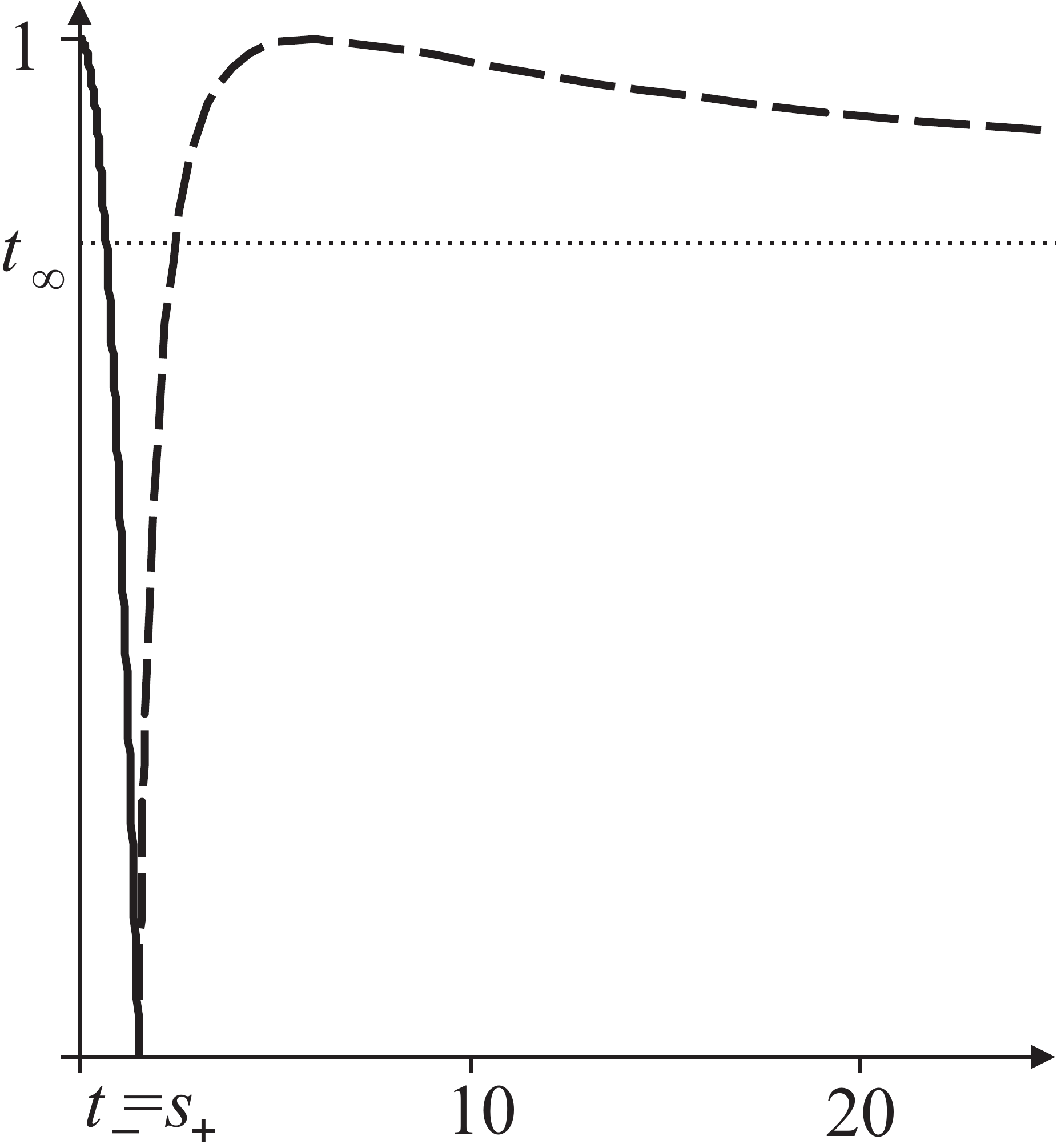}
 \caption{{\bf The Farey spectrum and the L\"uroth spectrum intersect in a single point, for $\alpha$ expansive.} The $\alpha$-Farey free energy $v$ (solid line),
 the $\alpha$-L\"uroth pressure function $p$ (dashed line), and the
 associated  dimension
 graphs for
 $a_n:=\zeta\left(5/4\right)^{-1}n^{-5/4}$.
 Here, $F_{\alpha}$ exhibits no phase transition. } \label{fig:PressureSpectrum4_5}
 \end{figure}

 \begin{figure}[h]\center
  \includegraphics[width=0.4\textwidth]{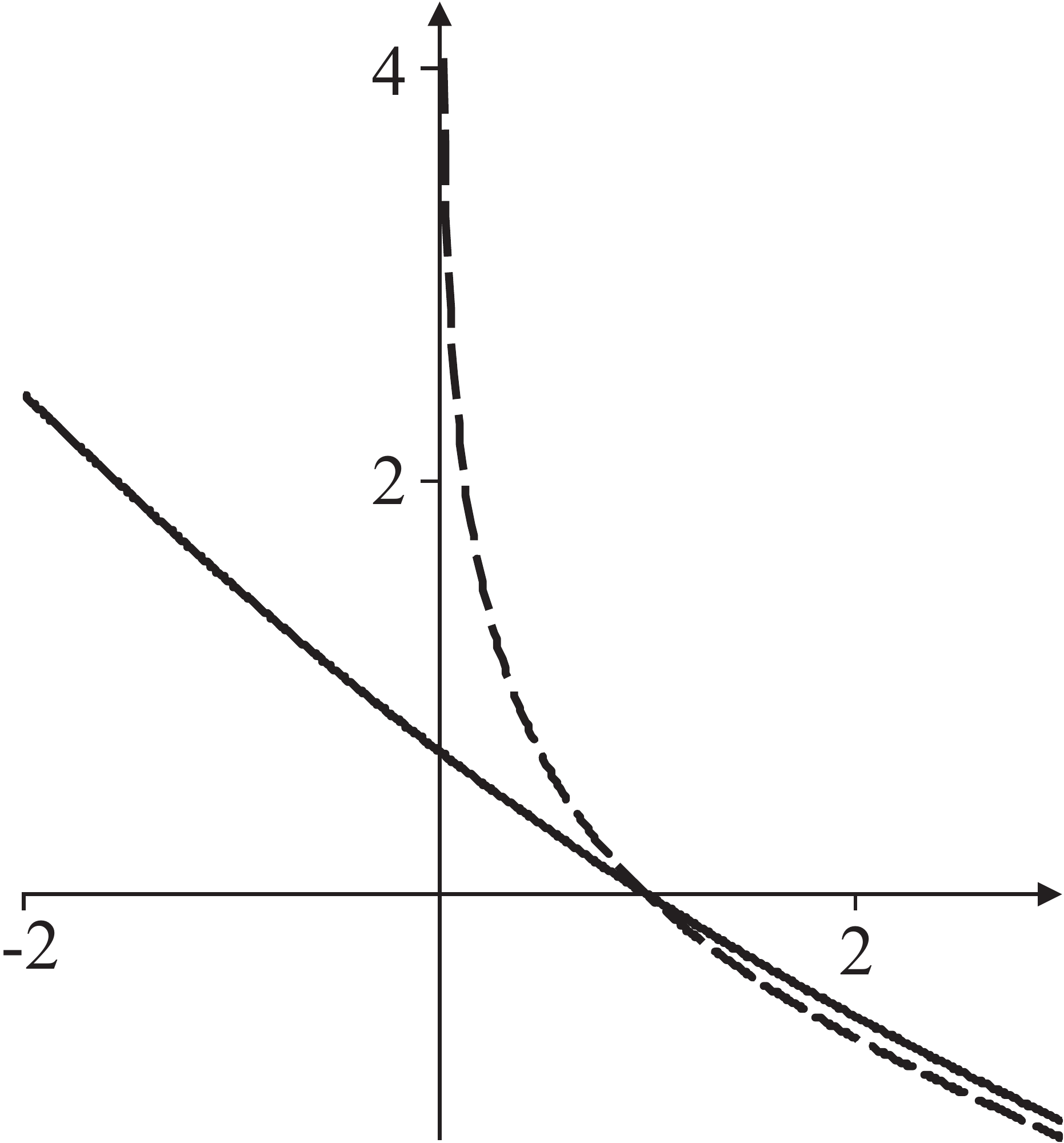}\hspace{0.1\textwidth}
  \includegraphics[width=0.4\textwidth]{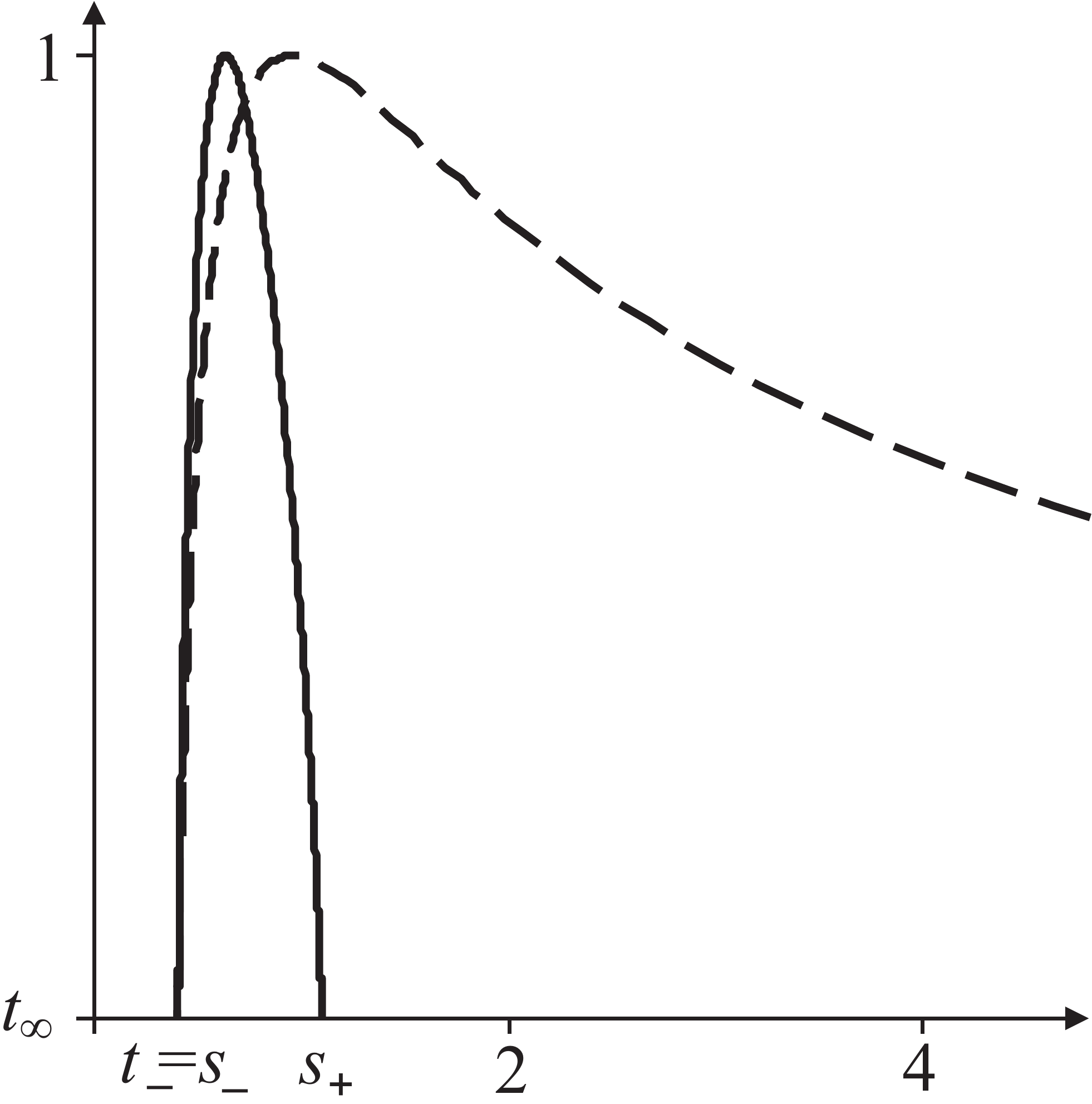}
  \caption{{\bf The Farey spectrum is completely contained in the L\"uroth spectrum, for $\alpha$ expanding.} The $\alpha$-Farey free energy $v$ (solid line),
  the $\alpha$-L\"uroth pressure function $p$ (dashed line), and the
  associated dimension
  graphs for the $\alpha$-Farey and $\alpha$-L\"uroth systems with $a_n:=2\cdot3^{-n}$, $n\in\N$.
  The $\alpha$-Farey system is given in this situation by the tent map
  with slopes $3$ and $-3/2$.}\label{fig:PressureSpectrumTent3_2}
  \end{figure}

Note that the H\"older exponent of the map $\theta_{\alpha_H}$ is
equal to $2\log 2/\log 6$. Also,
it is immediately clear that the invariant measure $\nu_{\alpha_{H}}$ associated to $F_{\alpha_H}$ is
infinite, and that the  density function of $\nu_{\alpha_{H}}$ with
respect to $\lambda$ is equal to the step function  $\sum_{n=1}^\infty (n+1)\1_{A_n}$.

\proc{Remark.}
Suppose that in the definition of $F_{\alpha_H}$ given above we were to choose
$x\mapsto 2x-1$ instead of $x\mapsto 2-2x$ for the right-hand branch of the
$\alpha_H$-Farey map, with $1/2\mapsto 0$. Then, the jump transformation of  this new
non-alternating version of $F_{\alpha_H}$  coincides
with the  actual classical L\"{u}roth map
${L}:\mathcal{U}\to \mathcal{U}$, which generates the series expansion of real numbers
introduced by L\"{u}roth in \cite{L} and which
is given  in our terms by
\[
L(x) = \left\{
        \begin{array}{ll} n(n+1)x-n, & \text{ for }x\in [1/(n+1),1/n), \ n \geq2;
        \\2x-1 , & \hbox{  for $x\in[1/2,1]$.}
        \\0 , & \hbox{  for $x=0$.}
        \end{array}
      \right.
\]
The series expansion in this case is given by
\[
x=\sum_{n=1}^\infty\left(\ell_n\prod_{k=1}^n(\ell_k(\ell_{k}+1))^{-1}\right),
\]
where again all of the $\ell_i$ are natural numbers. Notice that the atoms of the partition behind the map $L$ are slightly different to the atoms $A_n$ of $\alpha_H$, they are right closed and left open intervals, except for the equivalent of $A_1$, which is the closed interval $[1/2, 1]$.
\medbreak

 We now consider the $\alpha_H$-sum-level sets.
 The reader might like to see that the Lebesgue measures of the  first members of the sequence $\left(\mathcal{L}^{(\alpha_H)}_{n}\right)$
are as follows:

 \[
\lambda(\mathcal{L}^{(\alpha_H)}_{0})=1,\;\lambda(\mathcal{L}^{(\alpha_H)}_{1})=\frac{1}{2},\;\lambda(\mathcal{L}^{(\alpha_H)}_{2})=\frac{5}{12},\;\lambda(\mathcal{L}^{(\alpha_H)}_{3})=\frac{3}{8},\;\lambda(\mathcal{L}^{(\alpha_H)}_{4})=\frac{251}{720}.\]

 Since the Lebesgue measure of the sum-level set $\mathcal{L}_n$
 associated with the map $L$ coincides with the
 Lebesgue measure of the sum-level set ${\mathcal{L}}^{(\alpha_{H})}_{n}$,  Theorem \ref{renewal}  gives the following corollaries.
\begin{cor}
 $\displaystyle{\lim_{n \to \infty} \lambda
 \left(\mathcal{L}^{(\alpha_H)}_n\right)= \lim_{n \to \infty} \lambda
 \left(\mathcal{L}_n\right)=0}$.    \end{cor}

\begin{cor}
For the classical and for the alternating L\"uroth map  the following
hold,
  for $n$ tending to infinity.
\begin{enumerate}
  \item $ \displaystyle{\sum_{k=1}^n\lambda\left(\mathcal{ L}_n  \right)=\sum_{k=1}^n\lambda\left({\mathcal L}_n^{(\alpha_H)}
  \right)\sim n\left(\sum_{k=1}^n
  \frac1k\right)^{-1}\sim \frac{n}{\log n}}$;
  \item $\displaystyle{\lambda\left(\mathcal{L}_n\right)=
  \lambda\left({\mathcal L}_n^{(\alpha_H)}\right)
  \sim\left(\sum_{k=1}^n\frac1k\right)^{-1}\sim\frac1{\log n}}$.

\end{enumerate}
\end{cor}

For the outcome of the Lyapunov spectra associated with the harmonic
partition we refer to Fig. \ref{fig:PressureSpecCL}. Also, various different phenomena
which arise from particularly chosen partitions are briefly discussed in Fig.
\ref{fig:PressureSpectrum3}, \ref{fig:PressureSpectrumPhaseTransition},
\ref{fig:PressureSpectrumNoPhaseTransition},  \ref{fig:PressureSpectrum4_5},
\ref{fig:PressureSpectrumTent3_2} (see also Remark 2 in the
introduction).

\end{document}